\documentclass{ws-m3as}

\usepackage[T1]{fontenc}
\usepackage[super]{cite}
\usepackage{url}
\usepackage{times}
\usepackage[english]{babel}
\usepackage{times}
\usepackage{lipsum}

\def\ud{{\rm\,d}}

\def\R{\mathbb{R}}
\def\Sph{\mathbb{S}}

\def\CC{\mathcal{C}}
\def\DD{\mathcal{D}}

\def\GG{\mathcal{G}}
\def\LL{\mathcal{L}}

\def\OO{\mathcal{O}}

\def\QQ{\mathcal{Q}}

\newcommand{\mrm}[1]{\mathrm{#1}}
\newcommand{\bs}[1]{\boldsymbol{#1}}
\newcommand{\Stwo}{{\mathbb{S}^2}}
\def\xb{\bs{x}}
\def\yb{\bs{y}}

\def\Vb{\bs{V}}
\def\pr(#1){\left({#1}\right)}
\def\br[#1]{\left[{#1}\right]}

\def\abs#1{\left|{#1}\right|}
\def\norm#1{\left\|{#1}\right\|}
\def\conj#1{\overline{#1}}
\def\pFq#1#2{{\,}_{#1}F_{#2}}
\def\ii{{\rm i}}

\DeclareRobustCommand{\rchi}{{\mathpalette\irchi\relax}}
\newcommand{\irchi}[2]{\raisebox{\depth}{$#1\chi$}}

\begin{document}

\markboth{H.~Montanelli et al.}{Nonlocal vector calculus on the sphere}

%
\catchline{}{}{}{}{}
%
\title{Nonlocal vector calculus on the sphere}

\author{Hadrien Montanelli\footnote{Corresponding author}}
\address{Inria, Unit\'{e} de Math\'{e}matiques Appliqu\'{e}es, ENSTA, Institut Polytechnique de Paris, 91120 Palaiseau, France\\
hadrien.montanelli@inria.fr}

\author{Richard Mika\"el Slevinsky}
\address{Department of Mathematics, University of Manitoba, Winnipeg, R3T 5V6, Canada\\
richard.slevinsky@umanitoba.ca}

\author{Qiang Du}
\address{Department of Applied Physics and Applied Mathematics, and Data Science Institute, Columbia University, New York, NY 10027, USA\\
qd2125@columbia.edu}


\maketitle

\begin{history}
\received{(Day Month Year)}
\revised{(Day Month Year)}
\comby{(xxxxxxxxxx)}
\end{history}

\begin{abstract}
We introduce a nonlocal vector calculus on the unit two-sphere using weakly singular integral operators. Within this framework, the operators are diagonalizable in terms of scalar and vector spherical harmonics, a property that facilitates the proof of a nonlocal Stokes theorem. This constitutes the first instance of such a theorem on a curved surface. Furthermore, our analysis demonstrates the strong convergence of these nonlocal operators to the classical differential operators of vector calculus as the interaction range tends to zero.
\end{abstract}



\section{Introduction}

Nonlocal models are prevalent in applied fields such as materials science, thermodynamics, fluid dynamics, fracture mechanics, biology, and image analysis \cite{bates1999, bobaru2010, du2011, du2017a, du2019, gilboa2008, kao2010, silling2000}. In Ref.~\citen{du2013a}, Du, Gunzburger, Lehoucq, and Zhou developed a vector calculus for these operators in $\mathbb{R}^n$, termed a \textit{nonlocal vector calculus}, which parallels the standard vector calculus for differential operators. This paper goes beyond the concepts introduced in Ref.~\citen{du2013a} by developing a novel framework for nonlocal vector calculus on non-Euclidean domains. In particular, it establishes a nonlocal Stokes theorem on the unit two-sphere $\Stwo\subset\mathbb{R}^3$, which requires a careful redefinition of operators and integral identities in a nonlocal, curved setting. 

In a previous work \cite{montanelli2018b}, the authors introduced nonlocal Laplace--Beltrami operators on $\Stwo$ of the form
\begin{align}
\mathcal{L}_S^\delta\{u\}(\xb) = \int_{\Stwo}\left[u(\yb) - u(\xb)\right]\rho_\delta(\vert\xb - \yb\vert)\ud\Omega(\yb),
\label{eq:nonloc_diff}
\end{align}

\noindent where $\vert\xb - \yb\vert$ is the Euclidean distance in $\mathbb{R}^3$ between $\xb$ and $\yb$ on $\Stwo$, $\mathrm{d}\Omega(\yb)$ denotes the standard measure on $\Stwo$, and $\rho_\delta$ is a suitably defined nonlocal kernel with horizon $0 < \delta \leq 2$, determining the range of interactions.\footnote{An horizon $\delta = 0$ corresponds to a local operator, while an horizon $\delta = 2$---the maximum Euclidean distance between two points on $\Stwo$---corresponds to a global operator representing all-to-all coupling.} In this paper, we define nonlocal surface divergence, gradient, and curl operators and their adjoints on $\Stwo$ (Sec.~\ref{sec:nonloc_op}). Building on these definitions, we develop a nonlocal vector calculus using spherical harmonics, demonstrate that our nonlocal operators are diagonal when these are used as bases, and prove a nonlocal Stokes theorem (Sec.~\ref{sec:nonloc_stokes}). This incentivizes the adoption of the nonlocal operators defined by integrals studied here over other choices.\footnote{It is not feasible to develop a vector calculus for fractional operators via their spectral definition because the domain and the range of operators, such as the surface gradient, require different bases. For a review of the fractional Laplace operator in $\mathbb{R}^n$, see Ref.~\citen{kwasnicki2017}.} Finally, we prove that our nonlocal operators converge strongly to their local analogues as the horizon $\delta \rightarrow 0$ (Sec.~\ref{sec:nonloc_anal}).

\section{Nonlocal operators and their adjoints}\label{sec:nonloc_op}

Throughout the paper, we will use regular font to denote scalars and bold font to denote vectors. We identify a point $\xb\in\Stwo$ with the unit normal vector to $\Stwo$ at $\xb$. Therefore, points $\xb\in\Stwo$ correspond to vectors orthogonal to $\Stwo$. The Euclidean distance between $\xb,\yb\in\Stwo$ is given by
\begin{align}
\vert\xb-\yb\vert = \sqrt{2(1-\xb\cdot\yb)},
\label{eq:distance}
\end{align}
where $\cdot$ denotes the standard Euclidean inner product in $\R^3$. When necessary, we will use colatitude $\theta\in[0,\pi]$ and longitude $\varphi\in[0,2\pi)$ to place a point $\xb$ on the sphere, and $\ud\Omega = \sin\theta\ud\theta\ud\varphi$ for the Lebesgue measure on $\Stwo$.

The spherical basis, $\bs{e_\theta}$, $\bs{e_\varphi}$ and $\xb$, allows for a canonical decomposition of vector fields $\Vb:\Stwo\to\R^3$ at any $\xb\in\Stwo$,
\begin{align}
\Vb = V^\theta(\theta,\varphi)\bs{e_\theta} + V^\varphi(\theta,\varphi)\bs{e_\varphi} + V^x(\theta,\varphi)\xb,
\label{eq:spherical}
\end{align}
for some scalar fields $V^\theta,V^\varphi,V^x:\Stwo\to\R$. The first two components lie in the tangent plane to $\Stwo$ at $\xb$; these are therefore orthogonal to the third component, which is along the normal $\xb$. Another canonical decomposition of vector fields is the \textit{spheroidal-toroidal decomposition}, which reads
\begin{align}
\Vb = \nabla_S V^s(\theta,\varphi) + \bs{x}\times\nabla_S V^t(\theta,\varphi) + V^x(\theta,\varphi)\xb,
\label{eq:spheroidal}
\end{align}
where $\nabla_S$ and $\bs{x}\times\nabla_S$ are the surface gradient and vector surface curl of the scalar fields $V^s,V^t:\Stwo\to\R$, whose definitions are recalled below, together with the surface divergence, scalar surface curl, and curl of a vector field. We denote the \textit{tangent bundle} by $T\Stwo$, the set of tangent planes to $\Stwo$. For details, we refer to Refs.~\citen{atkinson2012} and \citen{nedelec2001}.

\begin{definition}[Local operators]\label{def:loc_operators}
Let $u:\Stwo\rightarrow\R$ and $\Vb:\Stwo\rightarrow\R^3$ be differentiable at any point $\xb\in\Stwo$ with the corresponding $\theta$ and $\varphi$. The surface divergence $\DD_S^0\{\Vb\}:\Stwo\rightarrow\R$ is the scalar field defined by
\begin{align}
\DD_S^0\{\Vb\} = \nabla_S\cdot\Vb = \csc\theta\partial_\theta(\sin\theta V^\theta) + \csc\theta\partial_\varphi V^\varphi.
\label{eq:surf_div}
\end{align}
The scalar surface curl $\CC_S^0\{\Vb\}:\Stwo\rightarrow\R$ is the scalar field defined by
\begin{align} 
\CC_S^0\{\Vb\} = \nabla_S\cdot(\xb\times\Vb) = -\csc\theta\partial_\theta(\sin\theta V^\varphi) + \csc\theta\partial_\varphi V^\theta.
\label{eq:scal_surf_curl}
\end{align}
The surface gradient $\bs{\GG}_S^0\{u\}:\Stwo\rightarrow T\Stwo$ is the tangential vector field defined by
\begin{align}
\bs{\GG}_S^0\{u\} = \nabla_S u = \partial_\theta u\bs{e_\theta} + \csc\theta\partial_\varphi u\bs{e_\varphi}.
\label{eq:surf_grad}
\end{align}
The vector surface curl $\bs{\CC}_S^0\{u\}:\Stwo\rightarrow T\Stwo$ is the tangential vector field defined by
\begin{align} 
\bs{\CC}_S^0\{u\} = \xb\times\nabla_S u = -\csc\theta\partial_\varphi u\bs{e_\theta} + \partial_\theta u\bs{e_\varphi}.
\label{eq:vect_surf_curl}
\end{align}
Finally, the curl $\bs{\CC}^0\{\Vb\}:\Stwo\rightarrow\R^3$ is the vector field defined by
\begin{align}
\bs{\CC}^0\{\Vb\} = \nabla\times \Vb = & \; [\csc\theta\partial_\varphi V^x - V^\varphi]\bs{e_\theta} + [V^\theta - \partial_\theta V^x]\bs{e_\varphi}\nonumber\\
& \; + [\csc\theta\partial_\theta(\sin\theta V^\varphi) - \csc\theta\partial_\varphi V^\theta]\xb.
\label{eq:curl}
\end{align}
\end{definition}

The operators of Definition \ref{def:loc_operators} satisfy the following properties \cite[Thm.~2.5.19]{nedelec2001}.

\begin{lemma}[Local vector calculus]\label{lem:loc_vect_calc}
For any (twice) differentiable $u:\Sph^2\to\R$ and differentiable $\bs{V}:\Sph^2\to\R^3$,
\begin{align*}
& \DD_S^0\{\bs{\GG}_S^0\{u\}\} = \LL_S^0\{u\}, && \text{i.e.}, && \nabla_S\cdot(\nabla_S u) = \Delta_S u,\\
& \DD_S^0\{\bs{\CC}_S^0\{u\}\} = 0, && \text{i.e.}, && \nabla_S\cdot(\xb\times\nabla_S u) = 0,\\
& \DD_S^0\{u\xb\} = 0, && \text{i.e.}, && \nabla_S\cdot(u\xb) = 0, \\
& \CC_S^0\{\bs{\GG}_S^0\{u\}\} = 0, && \text{i.e.}, && \nabla_S\cdot(\xb\times(\nabla_S u)) = 0,\\
& \CC_S^0\{\bs{\CC}_S^0\{u\}\} = -\LL_S^0\{u\}, && \text{i.e.}, && \nabla_S\cdot(\xb\times(\xb\times\nabla_S u)) = -\Delta_S u,\\
& \CC_S^0\{u\xb\} = 0, && \text{i.e.}, && \nabla_S\cdot(\xb\times(u\xb)) = 0,
\end{align*}
where $\mathcal{L}_S^0=\Delta_S$ is the Laplace--Betlrami operator defined by
\begin{align*}
\mathcal{L}_S^0\{u\} = \Delta_S u = u_{\theta\theta} + \csc\theta\cos\theta u_\theta + \csc^2\theta u_{\varphi\varphi}.
\end{align*}
Moreover,
\begin{align*}
\bs{\CC}^0\{\Vb\} = \xb\times\Vb - \bs{\CC}_S^0\{\xb\cdot\Vb\} - \CC_S^0\{\Vb\}\xb,
\end{align*}
i.e.,
\begin{align}
\nabla\times\Vb &= \xb\times\Vb - \xb\times\nabla_S(\xb\cdot\Vb) - [\nabla_S\cdot(\xb\times\Vb)]\xb.\label{eq:cross}
\end{align}
and, in particular,
\begin{align*}
& \bs{\CC}^0\{\bs{\GG}_S^0\{u\}\} = \bs{\CC}_S^0\{u\}, && \text{i.e.}, && \nabla\times(\nabla_S u) = \xb\times\nabla_S u,\\
& \bs{\CC}^0\{\bs{\CC}_S^0\{u\}\} = -\bs{\GG}_S^0\{u\} + \LL_S^0\{u\}\xb, && \text{i.e.}, && \nabla\times(\xb\times\nabla_S u) = -\nabla_S u + \Delta_S u\xb,\\
& \bs{\CC}^0\{u\xb\} = -\bs{\CC}_S^0\{u\}, && \text{i.e.}, && \nabla\times(u\xb) = -\xb\times\nabla_S u.
\end{align*}
Finally,
\begin{align*}
& \DD_S^0\{u\bs{V}\} = \bs{\GG}_S^0\{u\}\cdot\bs{V} + u\DD_S^0\{\bs{V}\},\\
& \bs{\CC}^0\{u\Vb\} = \bs{\GG}_S^0\{u\}\times\bs{V} + u\bs{\CC}^0\{\Vb\},
\end{align*}
i.e.,
\begin{align}
& \nabla_S\cdot(u\bs{V}) = \nabla_S u\cdot\bs{V} + u\nabla_S\cdot\bs{V},\label{eq:div-prod}\\
& \nabla\times(u\bs{V}) = \nabla_S u\times\bs{V} + u\nabla\times\bs{V}\label{eq:cross-prod}.
\end{align}
\end{lemma}

The operators of Definition \ref{def:loc_operators} also satisfy the following properties for continuously differentiable $u$ and $\Vb$,
\begin{align*}
& \int_{\Stwo}u\DD_S^0\{\Vb\}\ud\Omega = -\int_{\Stwo}\bs{\GG}_S^0\{u\}\cdot\Vb\ud\Omega,\\
& \int_{\Stwo}u\CC_S^0\{\Vb\}\ud\Omega = \int_{\Stwo}\bs{\CC}_S^0\{u\}\cdot\Vb\ud\Omega,
\end{align*}
which we prove in the proof of Theorem \ref{thm:loc_ajdoints} to identify the adjoints of the operators of Definition \ref{def:loc_operators}. We recall that operators $\QQ$ such as the surface divergence \eqref{eq:surf_div} and scalar surface curl \eqref{eq:scal_surf_curl} have adjoints $\QQ^*$ defined by
\begin{align*}
\int_{\Stwo}u\QQ\{\Vb\}\ud\Omega = \int_{\Stwo}\QQ^*\{u\}\cdot\Vb\ud\Omega \quad \forall u,\forall\Vb.
\end{align*}
Operators $\bs{\QQ}$ such as the surface gradient \eqref{eq:surf_grad} and vector surface curl \eqref{eq:vect_surf_curl} have adjoints $\bs{\QQ}^*$ defined by
\begin{align*}
\int_{\Stwo}\Vb\cdot\bs{\QQ}\{u\}\ud\Omega = \int_{\Stwo}\bs{\QQ}^*\{\Vb\}u\ud\Omega \quad \forall u,\forall\Vb.
\end{align*}
Finally, operators $\bs{\QQ}$ such as the curl \eqref{eq:curl} have adjoints $\bs{\QQ}^*$ defined by
\begin{align*}
\int_{\Stwo}\bs{W}\cdot\bs{\QQ}\{\Vb\}\ud\Omega = \int_{\Stwo}\bs{\QQ}^*\{\bs{W}\}\cdot\Vb\ud\Omega \quad \forall\Vb,\forall\bs{W}.
\end{align*}

\begin{theorem}[Local adjoints]\label{thm:loc_ajdoints}
Let $u:\Stwo\rightarrow\R$ and $\Vb:\Stwo\rightarrow\R^3$ be continuously differentiable. The adjoint surface divergence $(\DD_S^0)^*\{u\}:\Stwo\to T\Stwo$ is the tangential vector field defined by
\begin{align}
(\DD_S^0)^*\{u\} = -\partial_\theta u\bs{e_\theta} - \csc\theta\partial_\varphi u\bs{e_\varphi} = -\bs{\GG}_S^0\{u\}.
\label{eq:surf_div_adj}
\end{align}
The adjoint scalar surface curl $(\CC_S^0)^*\{u\}:\Stwo\rightarrow T\Stwo$ is the tangential vector field 
\begin{align} 
(\CC_S^0)^*\{u\} = -\csc\theta\partial_\varphi u\bs{e_\theta} + \partial_\theta u\bs{e_\varphi} = \bs{\CC}_S^0\{u\}.
\label{eq:scal_surf_curl_adj}
\end{align}
The adjoint surface gradient $(\bs{\GG}_S^0)^*\{\Vb\}:\Stwo\to\R$ is the scalar field defined by
\begin{align}
(\bs{\GG}_S^0)^*\{\Vb\} = \csc\theta\partial_\theta(\sin\theta V^\theta) + \csc\theta\partial_\varphi V^\varphi = -\DD_S^0\{\Vb\}.
\label{eq:surf_grad_adj}
\end{align}
The adjoint vector surface curl $(\bs{\CC}_S^0)^*\{\Vb\}:\Stwo\to\R$ is the scalar field defined by
\begin{align}
(\bs{\CC}_S^0)^*\{\Vb\} = -\csc\theta\partial_\theta(\sin\theta V^\varphi) + \csc\theta\partial_\varphi V^\theta = \CC_S^0\{\Vb\}.
\label{eq:vect_surf_curl_adj}
\end{align}
Finally, the adjoint curl $(\bs{\CC}^0)^*\{\Vb\}:\Stwo\to\R^3$ is the vector field defined by
\begin{align}
(\bs{\CC}^0)^*\{\Vb\} = & \; [\csc\theta\partial_\varphi V^x + V^\varphi]\bs{e_\theta} + [-V^\theta - \partial_\theta V^x]\bs{e_\varphi}\nonumber\\
& \; + [\csc\theta\partial_\theta(\sin\theta V^\varphi) - \csc\theta\partial_\varphi V^\theta]\xb,
\label{eq:curl_adj}
\end{align}
which, in general, is not equal to $\bs{\CC}^0\{\Vb\}$.
\end{theorem}

\begin{proof}
For the adjoint surface divergence \eqref{eq:surf_div_adj} and surface gradient \eqref{eq:surf_grad_adj}, one can directly note that
\begin{align}
\int_{\Stwo}u\nabla_S\cdot\Vb\ud\Omega = \int_{\Stwo}\nabla_S\cdot(u\Vb)\ud\Omega - \int_{\Stwo}\nabla_S u\cdot\Vb\ud\Omega = -\int_{\Stwo}\nabla_S u\cdot\Vb\ud\Omega,
\end{align}
for any continuously differentiable $u$ and $\Vb$, since $\Stwo$ has no boundary.

Similarly, for the adjoint scalar surface curl \eqref{eq:scal_surf_curl_adj} and vector surface curl \eqref{eq:vect_surf_curl_adj}, we note that
\begin{align}
\int_{\Stwo}u\nabla_S\cdot(\xb\times\Vb)\ud\Omega = -\int_{\Stwo}\nabla_S u\cdot(\xb\times\Vb)\ud\Omega = \int_{\Stwo}(\xb\times\nabla_S u)\cdot\Vb\ud\Omega,
\end{align}
for any continuously differentiable $u$ and $\Vb$.

Finally, for the adjoint curl \eqref{eq:curl_adj}, for any continuously differentiable $\Vb$ and $\bs{W}$,
\begin{align*}
& \; \int_{\Stwo}\bs{W}\cdot(\nabla_S\times\Vb)\ud\Omega\\
= & \; \int_{-\pi}^\pi\int_0^\pi\Big(W^\theta[\csc\theta\partial_\varphi V^x - V^\varphi] + W^\varphi[V^\theta - \partial_\theta V^x]\nonumber\\
& \; + W^x[\csc\theta\partial_\theta(\sin\theta V^\varphi) - \csc\theta\partial_\varphi V^\theta]\Big)\sin\theta\ud\theta\ud\varphi,\\
= & \; \int_{-\pi}^\pi\int_0^\pi\Big(V^\theta[\csc\theta\partial_\varphi W^x + W^\varphi] + V^\varphi[-W^\theta - \partial_\theta W^x]\nonumber\\
& \; + V^x[\csc\theta\partial_\theta(\sin\theta W^\varphi) - \csc\theta\partial_\varphi W^\theta]\Big)\sin\theta\ud\theta\ud\varphi.
\end{align*}
This completes the proof.
\end{proof}

We note that the adjoint curl satisfies
\begin{align*} 
(\bs{\CC}^0)^*\{\Vb\} = -\xb\times\Vb - \bs{\CC}_S^0\{\xb\cdot\Vb\} - \CC_S^0\{\Vb\}\xb,
\end{align*}
which implies
\begin{align*}
& (\bs{\CC}^0)^*\{\bs{\GG}_S^0\{u\}\} = -\bs{\CC}_S^0\{u\},\\
& (\bs{\CC}^0)^*\{\bs{\CC}_S^0\{u\}\} = \bs{\GG}_S^0\{u\} + \LL_S^0\{u\}\xb,\\
& (\bs{\CC}^0)^*\{u\xb\} = -\bs{\CC}_S^0\{u\}.
\end{align*}

Next, we \textit{formally} introduce \textit{nonlocal} analogues of the differential operators of Definition \ref{def:loc_operators}. That is, we first focus on their definitions for functions and kernels that are smooth on $\Stwo$, which can then be extended to more general classes of functions in Section \ref{sec:nonloc_anal}. While local operators act on scalar and vector fields defined on $\Stwo$ (\textit{one-point functions}), nonlocal operators act on scalar and vector fields defined on $\Stwo\times\Stwo$ (\textit{two-point functions}). This is similar to the Euclidean case \cite{du2013a}. However, in $\R^n$, nonlocal operators rely on \textit{antisymmetric} Euclidean vector fields $\bs{\beta}(\xb,\yb):\Stwo\times\Stwo\rightarrow\R^3$, that is, such that $\bs{\beta}(\xb,\yb) = -\bs{\beta}(\yb,\xb)$ for all $\xb,\yb\in\R^n$. On the sphere, we introduce the following concept.

\begin{definition}[Geodesic symmetry]
A geodesically symmetric vector field $\bs{\alpha}(\xb,\yb)$ is a vector field $\bs{\alpha}:\Stwo\times\Stwo\rightarrow T\Stwo$ that lives in the tangent plane at $\xb$, and verifies $\vert\bs{\alpha}(\xb,\yb)\vert = \vert\bs{\alpha}(\yb,\xb)\vert$ for all $\xb,\yb\in\Stwo$.
\end{definition}

We can construct a geodesically symmetric vector field $\bs{\alpha}$ from an antisymmetric vector field $\bs{\beta}:\Stwo\times\Stwo\rightarrow\R^3$ of the form $\bs{\beta}(\xb,\yb)=f(\xb\cdot\yb)[\yb-\xb]$ for some $f:[-1,1]\to\R$ by removing its component along $\xb$,
\begin{align}
\bs{\alpha}(\xb,\yb) = \bs{\beta}(\xb,\yb)-\left[\xb\cdot\bs{\beta}(\xb,\yb)\right]\xb.
\label{eq:geodkernel}
\end{align}
This is particularly important as the sphere may represent a coupling interface between physical phenomena inside and outside of the unit ball and one may be given a kernel $\bs{\beta}$. Let us also stress that antisymmetry implies geodesic symmetry, but the converse is not always true. Finally, for a geodesically symmetric vector field, the condition $\vert\bs{\alpha}(\xb,\yb)\vert = \vert\bs{\alpha}(\yb,\xb)\vert$ implies that $\vert\bs{\alpha}(\xb,\yb)\vert=\vert\xb\times\bs{\alpha}(\xb,\yb)\vert = \vert\yb\times\bs{\alpha}(\yb,\xb)\vert$ for all $\xb,\yb\in\Stwo$.

\begin{definition}[Nonlocal operators]\label{def:nonloc_operators}
Let $u:\Stwo\times\Stwo\rightarrow\R$ be a scalar field, $\Vb:\Stwo\times\Stwo\rightarrow\R^3$ a vector field, $\bs{\alpha}:\Stwo\times\Stwo\rightarrow T\Stwo$ a geodesically symmetric vector field, and $\bs{\beta}:\Stwo\times\Stwo\rightarrow\R^3$ an antisymmetric vector field. We assume that all functions are smooth. We define the nonlocal surface divergence $\DD_S\{\Vb\}:\Stwo\to\R$ as the scalar field
\begin{align}
\DD_S\{\Vb\}(\xb) = \int_{\Stwo}\left[\Vb(\xb,\yb)\cdot\bs{\alpha}(\xb,\yb) - \Vb(\yb,\xb)\cdot\bs{\alpha}(\yb,\xb)\right]\ud\Omega(\yb).
\label{eq:nonloc_surf_div}
\end{align}
We define the nonlocal scalar surface curl $\CC_S\{\Vb\}:\Stwo\to\R$ as the scalar field $\CC_S\{\Vb\} = \DD_S\{\xb\times\Vb\}$, i.e.,
\begin{align}
\CC_S\{\Vb\}(\xb) = \int_{\Stwo}\left\{\Vb(\yb,\xb)\cdot\left[\yb\times\bs{\alpha}(\yb,\xb)\right]-\Vb(\xb,\yb)\cdot\left[\xb\times\bs{\alpha}(\xb,\yb)\right]\right\}\ud\Omega(\yb).
\label{eq:nonloc_scal_surf_curl}
\end{align}
We define the nonlocal surface gradient $\bs{\GG}_S\{u\}:\Stwo\to T\Stwo$ as the tangential vector field
\begin{align}
\bs{\GG}_S\{u\}(\xb) = \int_{\Stwo}\left[u(\yb,\xb) - u(\xb,\yb)\right]\bs{\alpha}(\xb,\yb)\ud\Omega(\yb).
\label{eq:nonloc_surf_grad}
\end{align}
We define the nonlocal vector surface curl $\bs{\CC}_S\{u\}:\Stwo\to T\Stwo$ as the tangential vector field $\bs{\CC}_S\{u\} = \xb\times\bs{\GG}_S\{u\}$, i.e.,
\begin{align}
\bs{\CC}_S\{u\}(\xb) = \int_{\Stwo}\left[u(\yb,\xb) - u(\xb,\yb)\right]\xb\times\bs{\alpha}(\xb,\yb)\ud\Omega(\yb).
\label{eq:nonloc_vect_surf_curl}
\end{align}
Finally, we define the nonlocal curl $\bs{\CC}\{\Vb\}:\Stwo\to\R^3$ as the vector field
\begin{align}
\bs{\CC}\{\Vb\}(\xb) = \int_{\Stwo}\bs{\beta}(\xb,\yb)\times\left[\Vb(\yb,\xb) - \Vb(\xb,\yb)\right]\ud\Omega(\yb).
\label{eq:nonloc_curl}
\end{align}
\end{definition}

Let us now compute the adjoints of \eqref{eq:nonloc_surf_div}--\eqref{eq:nonloc_curl}. Nonlocal operators $\QQ$ such as the nonlocal surface divergence \eqref{eq:nonloc_surf_div} and nonlocal scalar surface curl \eqref{eq:nonloc_scal_surf_curl} have adjoints $\QQ^*$ acting on one-point scalar fields, resulting in two-point vector fields defined by
\begin{align*}
\int_{\Stwo}u(\xb)\QQ\{\Vb\}(\xb)\ud\Omega(\xb) = \iint_{\Stwo\times\Stwo}\QQ^*\{u\}(\xb,\yb)\cdot\Vb(\xb,\yb)\ud\Omega(\xb)\ud\Omega(\yb). 
\end{align*}
Nonlocal operators $\bs{\QQ}$ such as the nonlocal surface gradient \eqref{eq:nonloc_surf_grad} and nonlocal vector surface curl \eqref{eq:nonloc_vect_surf_curl} have adjoints $\bs{\QQ}^*$ acting on one-point vector fields, resulting in two-point scalar fields defined by
\begin{align*}
\int_{\Stwo}\Vb(\xb)\cdot\bs{\QQ}\{u\}(\xb)\ud\Omega(\xb) = \iint_{\Stwo\times\Stwo}\bs{\QQ}^*\{\Vb\}(\xb,\yb)u(\xb,\yb)\ud\Omega(\xb)\ud\Omega(\yb). 
\end{align*}
Finally, nonlocal operators $\bs{\QQ}$ such as the nonlocal curl \eqref{eq:nonloc_curl} have adjoints $\bs{\QQ}^*$ acting on one-point vector fields, resulting in two-point vector fields defined by
\begin{align*}
\int_{\Stwo}\bs{W}(\xb)\cdot\bs{\QQ}\{\Vb\}(\xb)\ud\Omega(\xb) = \iint_{\Stwo\times\Stwo}\bs{\QQ}^*\{\bs{W}\}(\xb,\yb)\cdot\Vb(\xb,\yb)\ud\Omega(\xb)\ud\Omega(\yb). 
\end{align*}
An important distinction between local and nonlocal operators is that the adjoints of the former involve the same set of operators (e.g., $(\DD_S^0)^*=-\GG_S^0$), while the adjoints of the latter involve different, \textit{nonintegral} operators.

We recall that, for any vector $\bs{a},\bs{b},\bs{c}\in\R^3$,
\begin{align}
& \bs{a}\cdot(\bs{b}\times\bs{c}) = \bs{b}\cdot(\bs{c}\times\bs{a}) = \bs{c}\cdot(\bs{a}\times\bs{b}),\label{eq:saclar-triple-prod}\\
& \bs{a}\times(\bs{b}\times\bs{c}) = (\bs{a}\cdot\bs{c})\bs{b} - (\bs{a}\cdot\bs{b})\bs{c},\label{eq:vector-triple-prod}\\
& [\bs{a}\cdot(\bs{b}\times\bs{c})]\bs{a} = (\bs{a}\times\bs{b})\times(\bs{a}\times\bs{c}).\label{eq:vector-triple-prod-2}
\end{align}

\begin{theorem}[Nonlocal adjoints]\label{thm:nonloc_ajdoints}
Let $u:\Stwo\rightarrow\R$ be a smooth scalar field and $\Vb:\Stwo\rightarrow\R^3$ a smooth vector field. The adjoint nonlocal surface divergence $\DD_S^*\{u\}:\Stwo\times\Stwo\to T\Stwo$ is the tangential vector field
\begin{align}
\DD_S^*\{u\}(\xb,\yb) = \left[u(\xb) - u(\yb)\right]\bs{\alpha}(\xb,\yb).
\label{eq:nonloc_surf_div_adj}
\end{align}
The adjoint nonlocal scalar surface curl $\CC_S^*\{u\}:\Stwo\times\Stwo\to T\Stwo$ is the tangential vector field
\begin{align}
\CC_S^*\{u\}(\xb,\yb) = \left[u(\yb)-u(\xb)\right]\xb\times\bs{\alpha}(\xb,\yb).
\label{eq:nonloc_scal_surf_curl_adj}
\end{align}
The adjoint nonlocal surface gradient $\bs{\GG}_S^*\{\Vb\}:\Stwo\times\Stwo\to\R$ is the scalar field
\begin{align}
\bs{\GG}_S^*\{\Vb\}(\xb,\yb) = \Vb(\yb)\cdot\bs{\alpha}(\yb,\xb) - \Vb(\xb)\cdot\bs{\alpha}(\xb,\yb).
\label{eq:nonloc_surf_grad_adj}
\end{align}
The adjoint nonlocal vector surface curl $\bs{\CC}_S^*\{\Vb\}:\Stwo\times\Stwo\to\R$ is the scalar field
\begin{align}
\bs{\CC}_S^*\{\Vb\}(\xb,\yb) = \Vb(\yb)\cdot\left[\yb\times\bs{\alpha}(\yb,\xb)\right] - \Vb(\xb)\cdot\left[\xb\times\bs{\alpha}(\xb,\yb)\right].
\label{eq:nonloc_vect_surf_curl_adj}
\end{align}
Finally, the adjoint nonlocal curl $\bs{\CC}^*\{\Vb\}:\Stwo\times\Stwo\to\R^3$ is the vector field
\begin{align}
\bs{\CC}^*\{\Vb\}(\xb,\yb) = \bs{\beta}(\xb,\yb)\times\left[\Vb(\xb) + \Vb(\yb)\right].
\label{eq:nonloc_curl_adj}
\end{align}
\end{theorem}

\begin{proof}
Let us start with the adjoint nonlocal surface divergence. By switching the variables in the integral, we have
\begin{align*} 
& \; \int_{\Stwo}u\DD_S\{\Vb\}\ud\Omega\\
= & \; \iint_{\Stwo\times\Stwo}u(\xb)\left[\Vb(\xb,\yb)\cdot\bs{\alpha}(\xb,\yb) - \Vb(\yb,\xb)\cdot\bs{\alpha}(\yb,\xb)\right]\ud\Omega(\yb)\ud\Omega(\xb),\\
= & \; \iint_{\Stwo\times\Stwo}\left[u(\xb)\Vb(\xb,\yb)\cdot\bs{\alpha}(\xb,\yb) - u(\yb)\Vb(\xb,\yb)\cdot\bs{\alpha}(\xb,\yb)\right]\ud\Omega(\yb)\ud\Omega(\xb),\\
= & \; \iint_{\Stwo\times\Stwo}\left[u(\xb)-u(\yb)\right]\bs{\alpha}(\xb,\yb)\cdot\Vb(\xb,\yb)\ud\Omega(\xb)\ud\Omega(\yb).
\end{align*}

For the adjoint nonlocal scalar surface curl, the same trick yields
\begin{align*} 
\int_{\Stwo}u\CC_S\{\Vb\}\ud\Omega = \iint_{\Stwo\times\Stwo}\left\{\left[u(\yb)-u(\xb)\right]\xb\times\bs{\alpha}(\xb,\yb)\right\}\cdot\Vb(\xb,\yb)\ud\Omega(\xb)\ud\Omega(\yb).
\end{align*}

Let us continue with the adjoint nonlocal surface gradient and the nonlocal vector surface curl. Similarly,
\begin{align*} 
\int_{\Stwo}\Vb\cdot\bs{\GG}_S\{u\}\ud\Omega = \iint_{\Stwo\times\Stwo}\left[\Vb(\yb)\cdot\bs{\alpha}(\yb,\xb)-\Vb(\xb)\cdot\bs{\alpha}(\xb,\yb)\right]u(\xb,\yb)\ud\Omega(\xb)\ud\Omega(\yb),
\end{align*}
and
\begin{align*} 
& \; \int_{\Stwo}\Vb\cdot\bs{\CC}_S\{u\}\ud\Omega\\
= & \; \iint_{\Stwo\times\Stwo}\left\{\Vb(\yb)\cdot\left[\yb\times\bs{\alpha}(\yb,\xb)\right]-\Vb(\xb)\cdot\left[\xb\times\bs{\alpha}(\xb,\yb)\right]\right\}u(\xb,\yb)\ud\Omega(\xb)\ud\Omega(\yb).
\end{align*}

We finish with the adjoint nonlocal curl. Using \eqref{eq:saclar-triple-prod}, we obtain
\begin{align*} 
& \; \int_{\Stwo}\bs{W}\cdot\bs{\CC}\{\Vb\}\ud\Omega\\
= & \; \iint_{\Stwo\times\Stwo}\bs{W}(\xb)\cdot\left[\bs{\beta}(\xb,\yb)\times\Vb(\yb,\xb)-\bs{\beta}(\xb,\yb)\times\Vb(\xb,\yb)\right]\ud\Omega(\yb)\ud\Omega(\xb),\\
= & \; \iint_{\Stwo\times\Stwo}\bs{\beta}(\xb,\yb)\times\left[\bs{W}(\xb)+\bs{W}(\yb)\right]\cdot\Vb(\xb,\yb)\ud\Omega(\xb)\ud\Omega(\yb).
\end{align*}
\end{proof}

We show next that, by composing nonlocal operators and their adjoints, we obtain ``second-order'' nonlocal operators, including the diffusion operator \eqref{eq:nonloc_diff}.

\begin{theorem}[Nonlocal compositions]\label{thm:nonloc_composition}
Let $u:\Stwo\rightarrow\R$ be a smooth scalar field and $\Vb:\Stwo\rightarrow T\Stwo$ a smooth vector field. The compositions of nonlocal operators and their adjoints are given by
\begin{align*}
\DD_S\{\DD_S^*\{u\}\}(\xb) = & \; 2\int_{\Stwo}\left[u(\xb)-u(\yb)\right]\vert\bs{\alpha}(\xb,\yb)\vert^2\ud\Omega(\yb),\\
\CC_S\{\CC_S^*\{u\}\}(\xb) = & \; 2\int_{\Stwo}\left[u(\xb)-u(\yb)\right]\vert\bs{\alpha}(\xb,\yb)\vert^2\ud\Omega(\yb),\\
\bs{\GG}_S\{\bs{\GG}_S^*\{\Vb\}\}(\xb) = & \; 2\int_{\Stwo}\left[\Vb(\xb)\cdot\bs{\alpha}(\xb,\yb)-\Vb(\yb)\cdot\bs{\alpha}(\yb,\xb)\right]\bs{\alpha}(\xb,\yb)\ud\Omega(\yb),\\
\bs{\CC}_S\{\bs{\CC}_S^*\{\Vb\}\}(\xb) = & \; 2\int_{\Stwo}\left\{\Vb(\xb)\cdot\left[\xb\times\bs{\alpha}(\xb,\yb)\right]-\Vb(\yb)\cdot\left[\yb\times\bs{\alpha}(\yb,\xb)\right]\right\}\\
& \; \xb\times\bs{\alpha}(\xb,\yb)\ud\Omega(\yb),\\
 \bs{\CC}\{\bs{\CC}^*\{\Vb\}\}(\xb) = & \; 2\int_{\Stwo}\bs{\beta}(\xb,\yb)\times\left\{\left[\Vb(\xb,\yb)+\Vb(\yb,\xb)\right]\times\bs{\beta}(\xb,\yb)\right\}\ud\Omega(\yb).
\end{align*}
The first two are scalar fields, the next two are tangential vector fields, while the last one is a vector field. Moreover, if we choose 
\begin{align}
2\vert\bs{\alpha}(\xb,\yb)\vert^2=\rho_\delta(\vert\xb-\yb\vert),
\label{eq:alpharho}
\end{align}
for some kernel $\rho_\delta$, then $\DD_S\{\DD_S^*\}=\CC_S\{\CC_S^*\}=-\mathcal{L}_S^\delta$, where $\mathcal{L}_S^\delta$ is the nonlocal diffusion operator \eqref{eq:nonloc_diff} with kernel $\rho_\delta$.
\end{theorem}

\begin{proof}
The five cases are nearly identical. For the divergence,
\begin{align*}
& \; \DD_S\{\DD_S^*\{u\}\}(\xb)\\
= & \; \int_{\Stwo}\left\{\left[u(\xb)-u(\yb)\right]\bs{\alpha}(\xb,\yb)\cdot\bs{\alpha}(\xb,\yb) - \left[u(\yb)-u(\xb)\right]\bs{\alpha}(\yb,\xb)\cdot\bs{\alpha}(\yb,\xb)\right\}\ud\Omega(\yb),\\
= & \; 2\int_{\Stwo}\left[u(\xb)-u(\yb)\right]\vert\bs{\alpha}(\xb,\yb)\vert^2\ud\Omega(\yb),
\end{align*}
where we have used geodesic symmetry $\vert\bs{\alpha}(\yb,\xb)\vert^2 = \vert\bs{\alpha}(\xb,\yb)\vert^2$. We immediately have that if $2\vert\bs{\alpha}(\xb,\yb)\vert^2=\rho_\delta(\vert\xb-\yb\vert)$ for some $\rho_\delta$, then $\DD_S\{\DD_S^*\}=\CC_S\{\CC_S^*\}=-\mathcal{L}_S^\delta$.
\end{proof}

Classical, local differential calculus deals with operators that maps one-point functions to other one-point functions. To demonstrate that our nonlocal operators really are analogous to the classical differential operators, we will use them to define corresponding nonlocal \textit{weighted} operators that map one-point functions to one-point functions.

\begin{definition}[Weighted nonlocal operators]\label{def:w_nonloc_operators}
Let $u:\Stwo\rightarrow\R$ be a smooth scalar field, $\Vb:\Stwo\rightarrow\R^3$ a smooth vector field, $\delta>0$, and $\omega_\delta:\Stwo\times\Stwo\to\R$ an integrable, non-negative radially symmetric function such that
\begin{align}
\omega_\delta(\xb,\yb) = 0 \quad \mrm{if} \quad \vert\xb-\yb\vert > \delta.
\end{align}
We define the weighted nonlocal surface divergence $\DD_S^\delta\{\Vb\}:\Stwo\to\R$ as the scalar field
\begin{align}\label{eq:w_nonloc_surf_div}
\DD_S^\delta\{\Vb\}(\xb) & = \DD_S\{\omega_\delta\Vb\}(\xb),\\
& = \int_{\Stwo}\left[\Vb(\xb)\cdot\bs{\alpha}(\xb,\yb) - \Vb(\yb)\cdot\bs{\alpha}(\yb,\xb)\right]\omega_\delta(\xb,\yb)\ud\Omega(\yb).\nonumber
\end{align}
We define the weighted nonlocal scalar surface curl $\CC_S^\delta\{\Vb\}:\Stwo\to\R$ as the scalar field
\begin{align}\label{eq:w_nonloc_saclar_surf_curl}
\CC_S^\delta\{\Vb\}(\xb) & = \CC_S\{\omega_\delta\Vb\}(\xb),\\
& = \int_{\Stwo}\left\{\Vb(\yb)\cdot\left[\yb\times\bs{\alpha}(\yb,\xb)\right]-\Vb(\xb)\cdot\left[\xb\times\bs{\alpha}(\xb,\yb)\right]\right\}\omega_\delta(\xb,\yb)\ud\Omega(\yb).\nonumber
\end{align}
We define the weighted nonlocal surface gradient $\bs{\GG}_S^\delta\{u\}:\Stwo\to T\Stwo$ as the tangential vector field
\begin{align}
\bs{\GG}_S^\delta\{u\}(\xb) = \bs{\GG}_S\{\omega_\delta u\}(\xb) = \int_{\Stwo}\left[u(\yb) - u(\xb)\right]\bs{\alpha}(\xb,\yb)\omega_\delta(\xb,\yb)\ud\Omega(\yb).
\label{eq:w_nonloc_surf_grad}
\end{align}
We define the weighted nonlocal vector surface curl $\bs{\CC}_S^\delta\{u\}:\Stwo\to T\Stwo$ as the tangential vector field
\begin{align}
\bs{\CC}_S^\delta\{u\}(\xb) = \bs{\CC}_S\{\omega_\delta u\}(\xb) = \int_{\Stwo}\left[u(\yb) - u(\xb)\right]\xb\times\bs{\alpha}(\xb,\yb)\omega_\delta(\xb,\yb)\ud\Omega(\yb).
\label{eq:w_nonloc_vector_surf_curl}
\end{align}
Finally, we define the weighted nonlocal curl $\bs{\CC}^\delta\{\Vb\}:\Stwo\to\R^3$ as the vector field
\begin{align}
\bs{\CC}^\delta\{\Vb\}(\xb) = \bs{\CC}\{\omega_\delta\Vb\}(\xb) = \int_{\Stwo}\bs{\beta}(\xb,\yb)\times\left[\Vb(\yb) - \Vb(\xb)\right]\omega_\delta(\xb,\yb)\ud\Omega(\yb).
\label{eq:w_nonloc_curl}
\end{align}
\end{definition}

Next, let us determine the adjoints of the weighted nonlocal operators.

\begin{theorem}[Weighted nonlocal adjoints]\label{thm:w_nonloc_ajdoints}
Let $u:\Stwo\rightarrow\R$ be a smooth scalar field and $\Vb:\Stwo\rightarrow\R^3$ a smooth vector field. The adjoint weighted nonlocal surface divergence $(\DD_S^\delta)^*\{u\}:\Stwo\to T\Stwo$ is the tangential vector field
\begin{align*}
(\DD_S^\delta)^*\{u\}(\xb) =\int_{\Stwo}\left[u(\xb)-u(\yb)\right]\bs{\alpha}(\xb,\yb)\omega_\delta(\xb,\yb)\ud\Omega(\yb) = -\bs{\GG}_S^\delta\{u\}(\xb).
\end{align*}
It also satisfies
\begin{align*}
(\DD_S^\delta)^*\{u\}(\xb) = \int_{\Stwo}\DD_S^*\{u\}(\xb,\yb)\omega_\delta(\xb,\yb)\ud\Omega(\yb).
\end{align*}
The adjoint weighted nonlocal scalar surface curl $(\CC_S^\delta)^*\{u\}:\Stwo\to T\Stwo$ is the tangential vector field
\begin{align*}
(\CC_S^\delta)^*\{u\}(\xb) =\int_{\Stwo}\left[u(\yb)-u(\xb)\right]\xb\times\bs{\alpha}(\xb,\yb)\omega_\delta(\xb,\yb)\ud\Omega(\yb) = \bs{\CC}_S^\delta\{u\}(\xb).
\end{align*}
It also satisfies
\begin{align*}
(\CC_S^\delta)^*\{u\}(\xb) = \int_{\Stwo}\CC_S^*\{u\}(\xb,\yb)\omega_\delta(\xb,\yb)\ud\Omega(\yb).
\end{align*}
The adjoint weighted nonlocal surface gradient $(\bs{\GG}_S^\delta)^*\{\Vb\}:\Stwo\to\Stwo$ is the scalar field
\begin{align*}
(\bs{\GG}_S^\delta)^*\{\Vb\}(\xb) & = \int_{\Stwo}\left[\Vb(\yb)\cdot\bs{\alpha}(\yb,\xb)-\Vb(\xb)\cdot\bs{\alpha}(\xb,\yb)\right]\omega_\delta(\xb,\yb)\ud\Omega(\yb),\\
& = -\DD_S^\delta\{\Vb\}(\xb).
\end{align*}
It also satisfies
\begin{align*} 
(\bs{\GG}_S^\delta)^*\{\Vb\}(\xb) = \int_{\Stwo}\bs{\GG}_S^*\{\Vb\}(\xb,\yb)\omega_\delta(\xb,\yb)\ud\Omega(\yb).
\end{align*}
The adjoint weighted nonlocal vector surface curl $(\bs{\CC}_S^\delta)^*\{\Vb\}:\Stwo\to\Stwo$ is the scalar field
\begin{align*}
(\bs{\CC}_S^\delta)^*\{\Vb\}(\xb) & = \int_{\Stwo}\left\{\Vb(\yb)\cdot\left[\yb\times\bs{\alpha}(\yb,\xb)\right]-\Vb(\xb)\cdot\left[\xb\times\bs{\alpha}(\xb,\yb)\right]\right\}\omega_\delta(\xb,\yb)\ud\Omega(\yb),\\
& = \CC_S^\delta\{\Vb\}(\xb).
\end{align*}
It also satisfies
\begin{align*} 
(\bs{\CC}_S^\delta)^*\{\Vb\}(\xb) = \int_{\Stwo}\bs{\CC}_S^*\{\Vb\}(\xb,\yb)\omega_\delta(\xb,\yb)\ud\Omega(\yb).
\end{align*}
Finally, the adjoint weighted nonlocal curl $(\bs{\CC}^\delta)^*\{\Vb\}:\Stwo\to\R^3$ is the vector field
\begin{align*}
(\bs{\CC}^\delta)^*\{\Vb\}(\xb) = \int_{\Stwo}\bs{\beta}(\xb,\yb)\times\left[\Vb(\xb)+\Vb(\yb)\right]\omega_\delta(\xb,\yb)\ud\Omega(\yb).
\end{align*}
which, in general, is not equal to $\bs{\CC}^\delta\{\Vb\}(\xb)$. It also satisfies
\begin{align*}
(\bs{\CC}^\delta)^*\{\Vb\}(\xb) = \int_{\Stwo}\bs{\CC}^*\{\Vb\}(\xb,\yb)\omega_\delta(\xb,\yb)\ud\Omega(\yb).
\end{align*}
\end{theorem}

\begin{proof}
We start with the adjoint weighted nonlocal surface divergence and gradient. We have
\begin{align*} 
& \; \int_{\Stwo}u\DD_S^\delta\{\Vb\}\ud\Omega\\
= & \; \iint_{\Stwo\times\Stwo}u(\xb)\left[\Vb(\xb)\cdot\bs{\alpha}(\xb,\yb) - \Vb(\yb)\cdot\bs{\alpha}(\yb,\xb)\right]\omega_\delta(\xb,\yb)\ud\Omega(\yb)\ud\Omega(\xb),\\
= & \; \int_{\Stwo}\left(\int_{\Stwo}\left[u(\xb)-u(\yb)\right]\bs{\alpha}(\xb,\yb)\omega_\delta(\xb,\yb)\ud\Omega(\yb)\right)\cdot\Vb(\xb)\ud\Omega(\xb).
\end{align*}
This yields $(\DD_S^\delta)^*\{u\} = -\bs{\GG}_S^\delta\{u\}$ and
\begin{align*}
(\DD_S^\delta)^*\{u\}(\xb) & = \int_{\Stwo}\left[u(\xb)-u(\yb)\right]\bs{\alpha}(\xb,\yb)\omega_\delta(\xb,\yb)\ud\Omega(\yb),\\
& = \int_{\Stwo}\DD_S^*\{u\}(\xb,\yb)\omega_\delta(\xb,\yb)\ud\Omega(\yb).
\end{align*}

For the adjoint weighted nonlocal scalar and vector surface curl, we have
\begin{align*} 
& \; \int_{\Stwo}\Vb\cdot\bs{\CC}_S^\delta\{u\}\ud\Omega\\
= & \; \int_{\Stwo}\int_{\Stwo}\left\{\Vb(\yb)\cdot\left[\yb\times\bs{\alpha}(\yb,\xb)\right]-\Vb(\xb)\cdot\left[\xb\times\bs{\alpha}(\xb,\yb)\right]\right\}\omega_\delta(\xb,\yb)\ud\Omega(\yb)u(\xb)\ud\Omega(\xb).
\end{align*}
This gives
\begin{align*} 
(\bs{\CC}_S^\delta)^*\{\Vb\}(\xb) = \int_{\Stwo}\left\{\Vb(\yb)\cdot\left[\yb\times\bs{\alpha}(\yb,\xb)\right]-\Vb(\xb)\cdot\left[\xb\times\bs{\alpha}(\xb,\yb)\right]\right\}\omega_\delta(\xb,\yb)\ud\Omega(\yb),
\end{align*}
and a short calculation also shows that $(\bs{\CC}_S^\delta)^*\{\Vb\}(\xb) = \DD_S^\delta\{\xb\times\Vb\}(\xb) = \CC_S^\delta\{\Vb\}(\xb)$.

We finish with the adjoint weighted nonlocal curl,
\begin{align*} 
& \; \int_{\Stwo}\bs{W}\cdot\bs{\CC}^\delta\{\Vb\}\ud\Omega\\
= & \; \iint_{\Stwo\times\Stwo}\bs{W}(\xb)\cdot\left[\bs{\beta}(\xb,\yb)\times\Vb(\yb)-\bs{\beta}(\xb,\yb)\times\Vb(\xb)\right]\omega_\delta(\xb,\yb)\ud\Omega(\yb)\ud\Omega(\xb),\\
= & \; \int_{\Stwo}\left(\int_{\Stwo}\bs{\beta}(\xb,\yb)\times\left[\bs{W}(\xb)+\bs{W}(\yb)\right]\omega_\delta(\xb,\yb)\ud\Omega(\yb)\right)\cdot\Vb(\xb)\ud\Omega(\xb).
\end{align*}
\end{proof}

Moving forward, we will restrict our attention to antisymmetric kernels
\begin{align}
\bs{\beta}(\xb,\yb) = \nabla^{\xb}\gamma(\xb\cdot\yb) - \nabla^{\yb}\gamma(\xb\cdot\yb) = \gamma'(\xb\cdot\yb)[\yb-\xb],
\label{eq:beta}
\end{align}
for some differentiable scalar kernel $\gamma$, and where $'$ denotes the derivative with respect to $t=\xb\cdot\yb$. (As in \eqref{eq:beta}, it is sometimes useful to disambiguate the gradient by superscripting it with the variable on which it acts.) This gives
\begin{align}
\bs{\alpha}(\xb,\yb) = \bs{\beta}(\xb,\yb) - \left[\bs{\beta}(\xb,\yb)\cdot\xb\right]\xb = \gamma'(\xb\cdot\yb)\left[\yb-(\xb\cdot\yb)\xb\right].
\label{eq:alpha}
\end{align}
We also choose $\omega_\delta(\xb,\yb) = \rchi_{[0,\delta]}(\vert\xb-\yb\vert)$, where $\rchi_{[0,\delta]}(\cdot)$ is the indicator function of the interval $[0,\delta]$, and set
\begin{align*}
& \gamma_\delta(\xb\cdot\yb) = \gamma(\xb\cdot\yb)\omega_\delta(\xb,\yb),\\
& \bs{\alpha}_\delta(\xb,\yb) = \gamma'_\delta(\xb\cdot\yb)\left[\yb-(\xb\cdot\yb)\xb\right],\\
& \bs{\beta}_\delta(\xb,\yb) = \gamma'_\delta(\xb\cdot\yb)(\yb-\xb).
\end{align*}
(We recall that the distance $\vert\xb-\yb\vert$ is only a function of the dot product $\xb\cdot\yb$; see Eq.~\eqref{eq:distance}.)

Kernels that depend solely on the dot product $\xb\cdot\yb$ exhibit properties the following properties; see, e.g., Ref.~\citen{brecht2016}.

\begin{lemma}\label{lem:kernel}
For any $\xb,\yb\in\Sph^2$,
\begin{align*}
& \nabla^{\xb}(\xb\cdot\yb) = \yb, && \nabla^{\yb}(\xb\cdot\yb) = \xb,\\
& \nabla^{\xb}_S(\xb\cdot\yb) = \yb - (\xb\cdot\yb)\xb, && \nabla^{\yb}_S(\xb\cdot\yb) = \xb - (\xb\cdot\yb)\yb,\\
& \vert\nabla^{\xb}_S(\xb\cdot\yb)\vert^2 = 1 - (\xb\cdot\yb)^2, && \vert\nabla^{\yb}_S(\xb\cdot\yb)\vert^2 = 1 - (\xb\cdot\yb)^2.
\end{align*}
For any differentiable kernel $f:\Sph^2\times\Sph^2\to\R$ and any $\xb,\yb\in\Sph^2$,
\begin{align*}
& \nabla^{\xb}f(\xb\cdot\yb) = f'(\xb\cdot\yb)\yb, && \nabla^{\yb}f(\xb\cdot\yb) = f'(\xb\cdot\yb)\xb,\\
& \nabla_S^{\xb}f(\xb\cdot\yb) = f'(\xb\cdot\yb)[\yb - (\xb\cdot\yb)\xb], && \nabla_S^{\yb}f(\xb\cdot\yb) = f'(\xb\cdot\yb)[\xb - (\xb\cdot\yb)\yb],\\
& \nabla_S^{\xb}f(\xb\cdot\yb)\cdot\yb = f'(\xb\cdot\yb)[1 - (\xb\cdot\yb)^2], && \nabla_S^{\yb}f(\xb\cdot\yb)\cdot\xb = f'(\xb\cdot\yb)[1 - (\xb\cdot\yb)^2],\\
& \xb\times\nabla_S^{\xb}f(\xb\cdot\yb) = f'(\xb\cdot\yb)\xb\times\yb, && \yb\times\nabla_S^{\yb}f(\xb\cdot\yb) = -f'(\xb\cdot\yb)\xb\times\yb.
\end{align*}
\end{lemma}

\section{Diagonalization and nonlocal Stokes theorem}\label{sec:nonloc_stokes}

Any scalar field in $L^2(\Stwo,\R)$ may be expanded in an orthonormal basis of spherical harmonics \cite{atkinson2012},
\begin{align}
u = \sum_{\ell=0}^{+\infty}\sum_{m=-\ell}^{+\ell} u_{\ell,m} Y_{\ell,m},
\label{eq:scalar_spherical_expansion}
\end{align}
for scalar coefficients $u_{\ell,m}$. The spherical harmonics are given by
\begin{align}
Y_{\ell,m} = \dfrac{e^{\ii m\varphi}}{\sqrt{2\pi}} \underbrace{\ii^{m+|m|}\sqrt{\left(\ell+\tfrac{1}{2}\right)\dfrac{(\ell-m)!}{(\ell+m)!}} P_\ell^m(\cos\theta)}_{\tilde{P}_\ell^m(\cos\theta)}, \quad \ell\geq0,\; -\ell\le m\le \ell,
\label{eq:spherical_harmonics}
\end{align}
where the notation $\tilde{P}_\ell^m$ for the associated Legendre polynomials is used to denote orthonormality for fixed $m$ in the $L^2(-1,1)$-sense, and verify 
\begin{align*}
\Delta_SY_{\ell,m}=-\ell(\ell+1)Y_{\ell,m}, \quad \ell\geq0,\; -\ell\le m\le \ell.
\end{align*}
We also define the three vector spherical harmonics \cite{barrera1985}
\begin{align}
\nabla_S Y_{\ell,m},\quad \xb\times\nabla_S Y_{\ell,m}, \quad Y_{\ell,m}\xb.
\label{eq:vector_spherical_harmonics}
\end{align}
Any vector field in $L^2(\Stwo,\R^3)$ may be expanded in an orthogonal basis of vector spherical harmonics
\begin{align}
\Vb = \sum_{\ell=0}^{+\infty}\sum_{m=-\ell}^{+\ell}\Big(V^s_{\ell,m}\nabla_S Y_{\ell,m} + V^t_{\ell,m}\xb\times\nabla_S Y_{\ell,m} + V^x_{\ell,m}Y_{\ell,m}\xb\Big),
\label{eq:vector_spherical_expansion}
\end{align}
for scalar coefficients $V_{\ell,m}^s$, $V_{\ell,m}^t$, and $V^x_{\ell,m}$ representing the spheroidal, toroidal, and normal components, respectively.

In contrast with the decomposition \eqref{eq:spherical} of a vector field using spherical basis vectors, the decomposition \eqref{eq:vector_spherical_expansion} is more elegant in the context of vector calculus, as the following theorem illustrates. 

\begin{theorem}[Diagonalization of local operators]\label{thm:local_operators}
The surface divergence and scalar surface curl satisfy
\begin{align*}
& \DD^0_S\{\nabla_S Y_{\ell,m}\} = -\ell(\ell+1)Y_{\ell,m}, && \CC^0_S\{\nabla_S Y_{\ell,m}\} = 0,\\
& \DD^0_S\{\xb\times\nabla_S Y_{\ell,m}\} = 0, && \CC^0_S\{\xb\times\nabla_S Y_{\ell,m}\} = \ell(\ell+1)Y_{\ell,m},\\
& \DD^0_S\{Y_{\ell,m}\xb\} = 0, && \CC^0_S\{Y_{\ell,m}\xb\} = 0.
\end{align*}
The surface gradient and vector surface curl satisfy by definition
\begin{align*}
& \bs{\GG}^0_S\{Y_{\ell,m}\} = \nabla_S Y_{\ell,m}, && \bs{\CC}^0_S\{Y_{\ell,m}\} = \xb\times\nabla_S Y_{\ell,m}.
\end{align*}
Finally, the curl and its adjoint satisfy
\begin{align*}
& \bs{\CC}^0\{\nabla_S Y_{\ell,m}\} = \xb\times\nabla_SY_{\ell,m},\\
& \bs{\CC}^0\{\xb\times\nabla_S Y_{\ell,m}\} = -\ell(\ell+1)Y_{\ell,m}\xb - \nabla_SY_{\ell,m},\\
& \bs{\CC}^0\{Y_{\ell,m}\xb\} = -\xb\times\nabla_SY_{\ell,m},\\
& (\bs{\CC}^0)^*\{\nabla_S Y_{\ell,m}\} = -\xb\times\nabla_SY_{\ell,m},\\
& (\bs{\CC}^0)^*\{\xb\times\nabla_S Y_{\ell,m}\} = -\ell(\ell+1)Y_{\ell,m}\xb + \nabla_SY_{\ell,m},\\
& (\bs{\CC}^0)^*\{Y_{\ell,m}\xb\} = -\xb\times\nabla_SY_{\ell,m}.
\end{align*}
\end{theorem}

We will show next that the weighted nonlocal operators may be reduced to diagonal scalings of their local counterparts. For any $f\in L^1(-1,1)$, define
\begin{align*}
\lambda_\ell\{f\} = 2\pi\int_{-1}^1P_\ell(t)f(t)\ud t, \quad \ell\ge0.
\end{align*}
We have the following \textit{localization} properties \cite{brecht2016}.

\begin{lemma}[Localization]\label{lem:localization}
Let $f\in L^1(-1,1)$ with $f'\in L^1(-1,1)$. Then for any $\xb\in\Stwo$
\begin{align*}
& \int_{\Sph^2}f'(\xb\cdot\yb)\yb\ud\Omega(\yb) = \lambda_0\{tf'\}\xb,\\
& \int_{\Sph^2}f'(\xb\cdot\yb)Y_{\ell,m}(\yb)\ud\Omega(\yb) = \lambda_\ell\{f'\}Y_{\ell,m}(\xb),\\
& \int_{\Sph^2}f'(\xb\cdot\yb)\Delta_S Y_{\ell,m}(\yb)\ud\Omega(\yb) = \lambda_\ell\{f'\}\Delta_S Y_{\ell,m}(\xb),\\
& \int_{\Sph^2}f'(\xb\cdot\yb)\nabla_SY_{\ell,m}(\yb)\ud\Omega(\yb) = \ell(\ell+1)\lambda_\ell\{f\}Y_{\ell,m}(\xb)\xb + \lambda_\ell\{f + tf'\}\nabla_S Y_{\ell,m}(\xb),\\
& \int_{\Sph^2}f'(\xb\cdot\yb)\yb\times\nabla_S Y_{\ell,m}(\yb)\ud\Omega(\yb) = \lambda_\ell\{f'\}\xb\times\nabla_S Y_{\ell,m}(\xb),\\
& \int_{\Sph^2}f'(\xb\cdot\yb)Y_{\ell,m}(\yb)\yb\ud\Omega(\yb) = \lambda_\ell\{tf'\}Y_{\ell,m}(\xb)\xb + \lambda_\ell\{f\}\nabla_S Y_{\ell,m}(\xb).
\end{align*}
\end{lemma}

We prove the following lemma, which be useful for what follows.

\begin{lemma}
Let $f\in L^1(-1,1)$ such that $\sqrt{1-t^2}f'\in L^1(-1,1)$. Then for any $\xb\in\Stwo$
\begin{align}
\int_{\Sph^2}\nabla^{\yb}_S f(\xb\cdot\yb)\times\nabla_SY_{\ell,m}(\yb)\ud\Omega(\yb) = & \; \lambda_\ell\{f\}\xb\times\nabla_SY_{\ell,m}(\xb),\nonumber\\ 
= & \; \lambda_\ell\{f\}\nabla\times(\nabla_SY_{\ell,m}(\xb)),\label{eq:12}\\
\int_{\Sph^2}\nabla^{\xb}_S f(\xb\cdot\yb)\times(\yb\times\nabla_S Y_{\ell,m}(\yb))\ud\Omega(\yb) = & \; - \lambda_\ell\{f\}\nabla_SY_{\ell,m}(\xb)\nonumber\\
& \;- \ell(\ell+1)\lambda_\ell\{f\}Y_{\ell,m}(\xb)\xb,\nonumber\\
= & \; \lambda_\ell\{f\}\nabla\times(\xb\times\nabla_SY_{\ell,m}(\xb)),\label{eq:13}\\
\int_{\Sph^2}\nabla^{\yb}_S f(\xb\cdot\yb)\times(Y_{\ell,m}(\yb)\yb)\ud\Omega(\yb) = & \; \lambda_\ell\{f\}\xb\times\nabla_SY_{\ell,m}(\xb),\nonumber\\
= & \; -\lambda_\ell\{f\}\nabla\times(Y_{\ell,m}\xb).\label{eq:14}
\end{align}
\end{lemma}

\begin{proof}
The last equality in each equation follows from Theorem \ref{thm:local_operators}. We prove the first equalities, starting with \eqref{eq:14},
\begin{align}
\int_{\Sph^2}\nabla^{\yb}_S f(\xb\cdot\yb)\times(Y_{\ell,m}(\yb)\yb)\ud\Omega(\yb) & = -\int_{\Sph^2}\yb\times\nabla^{\yb}_S f(\xb\cdot\yb)Y_{\ell,m}(\yb)\ud\Omega(\yb),\nonumber\\ 
& = \int_{\Sph^2} f(\xb\cdot\yb)\yb\times\nabla^{\yb}_S Y_{\ell,m}(\yb)\ud\Omega(\yb),
\end{align}
by integration by parts. We then use localization from Lemma \ref{lem:localization} for $\yb\times\nabla_SY_{\ell,m}(\yb)$.

For \eqref{eq:13}, we start with the right-hand side and use \eqref{eq:14},
\begin{align*}
\lambda_\ell\{f\}\nabla\times(\xb\times\nabla_SY_{\ell,m}(\xb)) & = \nabla^{\xb}\times\left(\int_{\Sph^2}\nabla^{\yb}_S f(\xb\cdot\yb)\times(Y_{\ell,m}(\yb)\yb)\ud\Omega(\yb)\right),\\
& = \nabla^{\xb}\times\left(\int_{\Sph^2} f(\xb\cdot\yb)\yb\times\nabla^{\yb}_S Y_{\ell,m}(\yb)\ud\Omega(\yb)\right).
\end{align*}
We insert the curl inside and then use \eqref{eq:cross-prod},
\begin{align}
& \; \nabla^{\xb}\times\left(\int_{\Sph^2} f(\xb\cdot\yb)\yb\times\nabla^{\yb}_S Y_{\ell,m}(\yb)\ud\Omega(\yb)\right)\nonumber\\ 
= & \; \int_{\Sph^2} \nabla^{\xb}_Sf(\xb\cdot\yb)\times(\yb\times\nabla^{\yb}_S Y_{\ell,m}(\yb))\ud\Omega(\yb).
\end{align}

The proof of \eqref{eq:12} is more involved; we do it in three steps. First, we show that 
\begin{align}
& \; \int_{\Sph^2}\nabla^{\yb}_S f(\xb\cdot\yb)\times\nabla_SY_{\ell,m}(\yb)\ud\Omega(\yb)\nonumber\\ 
= & \; \int_{\Sph^2}(\yb\times\nabla^{\yb}_S f(\xb\cdot\yb))\times(\yb\times\nabla_SY_{\ell,m}(\yb))\ud\Omega(\yb).
\end{align}
Second, we show that
\begin{align}
& \; \int_{\Sph^2}(\yb\times\nabla^{\yb}_S f(\xb\cdot\yb))\times(\yb\times\nabla_SY_{\ell,m}(\yb))\ud\Omega(\yb)\nonumber\\ 
= & \; -\int_{\Sph^2}\left[\xb\cdot(\yb\times\nabla_S Y_{\ell,m}(\yb))\right]\nabla_S^{\xb} f(\xb\cdot\yb)\ud\Omega(\yb).
\end{align}
Third, we show that we also have
\begin{align}
& \; \lambda_\ell\{f\}\xb\times\nabla_SY_{\ell,m}(\xb)\nonumber\\ 
= & \; -\int_{\Sph^2}\left[\xb\cdot(\yb\times\nabla_S Y_{\ell,m}(\yb))\right]\nabla_S^{\xb} f(\xb\cdot\yb)\ud\Omega(\yb).
\end{align}
For the first step, we utilize \eqref{eq:cross-prod},
\begin{align*}
\nabla^{\yb}\times(f(\xb\cdot\yb)\nabla_SY_{\ell,m}(\yb)) = \nabla_S f(\xb\cdot\yb)\times\nabla_SY_{\ell,m}(\yb) + f(\xb\cdot\yb)\yb\times\nabla_SY_{\ell,m}(\yb),
\end{align*}
and \eqref{eq:cross},
\begin{align}
\nabla^{\yb}\times(f(\xb\cdot\yb)\nabla_SY_{\ell,m}(\yb)) = & \; \yb\times f(\xb\cdot\yb)\nabla_SY_{\ell,m}(\yb)\nonumber\\ 
& \; - [\nabla^{\yb}_S\cdot(\yb\times f(\xb\cdot\yb)\nabla_SY_{\ell,m}(\yb))]\yb,
\end{align}
which yields, via \eqref{eq:div-prod},
\begin{align}
\nabla^{\yb}_S f(\xb\cdot\yb)\times\nabla_SY_{\ell,m}(\yb) & = -[\nabla^{\yb}_S\cdot(\yb\times f(\xb\cdot\yb)\nabla_SY_{\ell,m}(\yb))]\yb,\nonumber\\ 
& = -[\nabla^{\yb}_S f(\xb\cdot\yb)\cdot(\yb\times\nabla_SY_{\ell,m}(\yb))]\yb.
\end{align}
We conclude with \eqref{eq:vector-triple-prod-2},
\begin{align}
-[\nabla^{\yb}_S f(\xb\cdot\yb)\cdot(\yb\times\nabla_SY_{\ell,m}(\yb))]\yb = (\yb\times\nabla^{\yb}_S f(\xb\cdot\yb))\times(\yb\times\nabla_SY_{\ell,m}(\yb)).
\end{align}
For the second step, we write
\begin{align}
& \; \int_{\Sph^2}(\yb\times\nabla^{\yb}_S f(\xb\cdot\yb))\times(\yb\times\nabla_SY_{\ell,m}(\yb))\ud\Omega(\yb)\nonumber\\ 
= & \; -\int_{\Sph^2}(\yb\times\nabla_SY_{\ell,m}(\yb))\times(\nabla^{\xb}_S f(\xb\cdot\yb)\times\xb)\ud\Omega(\yb),
\end{align}
and utilize \eqref{eq:vector-triple-prod} to obtain
\begin{align}
& \; -\int_{\Sph^2}(\yb\times\nabla_SY_{\ell,m}(\yb))\times(\nabla^{\xb}_S f(\xb\cdot\yb)\times\xb)\ud\Omega(\yb)\nonumber\\ 
= & \; -\int_{\Sph^2}\left[\xb\cdot(\yb\times\nabla_S Y_{\ell,m}(\yb))\right]\nabla_S^{\xb} f(\xb\cdot\yb)\ud\Omega(\yb).
\end{align}
For the third step, we cross \eqref{eq:13} by $-\xb$,
\begin{align*}
\lambda_\ell\{f\}\xb\times\nabla_SY_{\ell,m}(\xb) = -\xb\times\left(\int_{\Sph^2}\nabla^{\xb}_S f(\xb\cdot\yb)\times(\yb\times\nabla_S Y_{\ell,m}(\yb))\ud\Omega(\yb)\right),
\end{align*}
and conclude with \eqref{eq:vector-triple-prod}.
\end{proof}

We now prove the following \textit{generalized} localization properties.

\begin{lemma}[Generalized localization]\label{lem:gen_localization}
Let $f\in L^1(-1,1)$ such that $\sqrt{1-t^2}f'\in L^1(-1,1)$. Then for any $\xb\in\Stwo$
\begin{align}
& \; \int_{\Stwo}f'(\xb\cdot\yb)(\yb-\xb)Y_{\ell,m}(\yb)\ud\Omega(\yb)\nonumber\\
= & \; \lambda_\ell\{(t-1)f'\}Y_{\ell,m}(\xb)\xb + \lambda_\ell\{f\}\nabla_S Y_{\ell,m}(\xb),\label{eq:Ylm}\\
& \; \int_{\Stwo}f'(\xb\cdot\yb)(\yb-\xb)\times\nabla_SY_{\ell,m}(\yb)\ud\Omega(\yb)\nonumber\\
= & \; \lambda_\ell\{(1-t)f' - f\}\xb\times\nabla_SY_{\ell,m}(\xb),\label{eq:Glm}\\
& \; \int_{\Stwo}f'(\xb\cdot\yb)(\yb-\xb)\times(\yb\times\nabla_SY_{\ell,m}(\yb))\ud\Omega(\yb)\nonumber\\
= & \; \lambda_\ell\{f\}\Delta_S Y_{\ell,m}(\xb)\xb + \lambda_\ell\{(1-t)f' - f\}\nabla_SY_{\ell,m}(\xb),\label{eq:Clm}\\
& \; \int_{\Stwo}f'(\xb\cdot\yb)(\yb-\xb)\times(Y_{\ell,m}(\yb)\yb)\ud\Omega(\yb) = -\lambda_\ell\{f\}\xb\times\nabla_SY_{\ell,m}(\xb).\label{eq:Xlm}
\end{align}
\end{lemma}

\begin{proof}
We start with \eqref{eq:Ylm}. We use Lemma \ref{lem:kernel} to write $\bs{y}=(\xb\cdot\yb)\xb + \nabla^{\xb}_S(\xb\cdot\yb)$ and $f'(\xb\cdot\yb)\nabla^{\xb}_S(\xb\cdot\yb)=\nabla^{\xb}_Sf(\xb\cdot\yb)$,
\begin{align}
& \; \int_{\Sph^2}f'(\xb\cdot\yb)(\yb-\xb)Y_{\ell,m}(\yb)\ud\Omega(\yb),\nonumber\\
= & \; \int_{\Sph^2}f'(\xb\cdot\yb)[(\xb\cdot\yb) - 1]Y_{\ell,m}(\yb)\ud\Omega(\yb)\xb + \int_{\Sph^2}\nabla^{\xb}_Sf(\xb\cdot\yb)Y_{\ell,m}(\yb)\ud\Omega(\yb).
\end{align}
This is where we need $\sqrt{1-t^2}f'\in L^1(-1,1)$ since 
\begin{align}
\vert\nabla^{\xb}_Sf(\xb\cdot\yb)\vert=\vert f'(\xb\cdot\yb)\vert\sqrt{1-(\xb\cdot\yb)^2},
\end{align}
via Lemma \ref{lem:kernel}. We conclude with localization from Lemma \ref{lem:localization} for $Y_{\ell,m}(\yb)$.

We continue with \eqref{eq:Glm}. Similarly, writing $\bs{x}=(\xb\cdot\yb)\yb + \nabla^{\yb}_S(\xb\cdot\yb)$,
\begin{align}
& \; \int_{\Sph^2}f'(\xb\cdot\yb)(\yb-\xb)\times\nabla_SY_{\ell,m}(\yb)\ud\Omega(\yb),\nonumber\\
= & \; \int_{\Sph^2}f'(\xb\cdot\yb)[1-(\xb\cdot\yb)]\yb\times\nabla_SY_{\ell,m}(\yb)\ud\Omega(\yb)\nonumber\\
& \; - \int_{\Sph^2}\nabla^{\yb}_S f(\xb\cdot\yb)\times\nabla_SY_{\ell,m}(\yb)\ud\Omega(\yb).
\end{align}
We conclude with localization from Lemma \ref{lem:localization} for $\yb\times\nabla_SY_{\ell,m}(\yb)$ and \eqref{eq:12}.

Let us now prove \eqref{eq:Clm}. Again, we write $\bs{y}=(\xb\cdot\yb)\xb + \nabla^{\xb}_S(\xb\cdot\yb)$,
\begin{align}
& \; \int_{\Sph^2}f'(\xb\cdot\yb)(\yb-\xb)\times(\yb\times\nabla_S Y_{\ell,m}(\yb))\ud\Omega(\yb),\nonumber\\
= & \; \xb\times\left(\int_{\Sph^2}f'(\xb\cdot\yb)[(\xb\cdot\yb) - 1]\yb\times\nabla_S Y_{\ell,m}(\yb)\ud\Omega(\yb)\right)\nonumber\\
& \; + \int_{\Sph^2}\nabla^{\yb}_S f(\xb\cdot\yb)\times(\yb\times\nabla_S Y_{\ell,m}(\yb))\ud\Omega(\yb).
\end{align}
We conclude with localization from Lemma \ref{lem:localization} for $\yb\times\nabla_SY_{\ell,m}(\yb)$ and \eqref{eq:13}.

We finish with \eqref{eq:Xlm}. We have
\begin{align}
& \; \int_{\Sph^2}f'(\xb\cdot\yb)(\yb-\xb)\times(Y_{\ell,m}(\yb)\yb)\ud\Omega(\yb),\nonumber\\
= & \; \int_{\Sph^2}f'(\xb\cdot\yb)[1-(\xb\cdot\yb)]\yb\times(Y_{\ell,m}(\yb)\yb)\ud\Omega(\yb)\nonumber\\
& \; - \int_{\Sph^2}\nabla^{\yb}_S f(\xb\cdot\yb)\times(Y_{\ell,m}(\yb)\yb)\ud\Omega(\yb).
\end{align}
We conclude with \eqref{eq:14}.
\end{proof}

We are now ready to prove our main results about nonlocal operators.

\begin{theorem}[Diagonalization of nonlocal operators]\label{thm:nonlocal_operators}
Let $\gamma_\delta\in L^1(-1,1)$ such that $\sqrt{1-t^2}\gamma_\delta'\in L^1(-1,1)$, and define
\begin{align}
\Lambda_\ell^\delta = \lambda_\ell\{\gamma_\delta\}, \quad \Theta_\ell^\delta = \lambda_\ell\{(1-t)\gamma_\delta'\}.
\end{align}
The weighted nonlocal surface divergence and scalar surface curl satisfy
\begin{align}
& \DD^\delta_S\{\nabla_S Y_{\ell,m}\} = -\Lambda_\ell^\delta\ell(\ell+1)Y_{\ell,m}, && \CC^\delta_S\{\nabla_S Y_{\ell,m}\} = 0,\\
& \DD^\delta_S\{\xb\times\nabla_S Y_{\ell,m}\} = 0, && \CC^\delta_S\{\xb\times\nabla_S Y_{\ell,m}\} = \Lambda_\ell^\delta\ell(\ell+1)Y_{\ell,m},\\
& \DD^\delta_S\{Y_{\ell,m}\xb\} = 0, && \CC^\delta_S\{Y_{\ell,m}\xb\} = 0.
\end{align}
The weighted nonlocal surface gradient and vector surface curl satisfy
\begin{align}
& \bs{\GG}^\delta_S\{Y_{\ell,m}\} = \Lambda_\ell^\delta\nabla_S Y_{\ell,m}, && \bs{\CC}^\delta_S\{Y_{\ell,m}\} = \Lambda_\ell^\delta\xb\times\nabla_S Y_{\ell,m}.
\end{align}
Finally, the weighted nonlocal curl and its adjoint satisfy
\begin{align}
& \bs{\CC}^\delta\{\nabla_S Y_{\ell,m}\} = [\Theta_0^\delta + \Theta_\ell^\delta - \Lambda_\ell^\delta]\xb\times\nabla_SY_{\ell,m},\\
& \bs{\CC}^\delta\{\xb\times\nabla_S Y_{\ell,m}\} = -\ell(\ell+1)\Lambda_\ell^\delta Y_{\ell,m}\xb + [-\Theta_0^\delta + \Theta_\ell^\delta - \Lambda_\ell^\delta]\nabla_SY_{\ell,m},\\
& \bs{\CC}^\delta\{Y_{\ell,m}\xb\} = -\Lambda_\ell^\delta\xb\times\nabla_SY_{\ell,m},\\
& (\bs{\CC}^\delta)^*\{\nabla_S Y_{\ell,m}\} = [-\Theta_0^\delta + \Theta_\ell^\delta - \Lambda_\ell^\delta]\xb\times\nabla_SY_{\ell,m},\\
& (\bs{\CC}^\delta)^*\{\xb\times\nabla_S Y_{\ell,m}\} = -\ell(\ell+1)\Lambda_\ell^\delta Y_{\ell,m}\xb + [\Theta_0^\delta + \Theta_\ell^\delta - \Lambda_\ell^\delta]\nabla_SY_{\ell,m},\\
& (\bs{\CC}^\delta)^*\{Y_{\ell,m}\xb\} = -\Lambda_\ell^\delta\xb\times\nabla_SY_{\ell,m}.
\end{align}
\end{theorem}

\begin{proof}
We begin with $\DD^\delta_S$. For $\Vb=Y_{\ell,m}\xb$, it is clear that the divergence vanishes. For $\Vb=\xb\times\nabla_S Y_{\ell,m}$,
\begin{align}
\DD_\delta\{\xb\times\nabla_S Y_{\ell,m}\}(\xb) = & \; \int_{\Stwo}\Big\{\left[\xb\times\nabla_S Y_{\ell,m}(\xb)\right]\cdot \yb\gamma_\delta'(\xb\cdot\yb)\nonumber\\
& \; - \left[\yb\times\nabla_S Y_{\ell,m}(\yb)\right]\cdot\nabla_S^{\yb}\gamma_\delta(\xb\cdot\yb)\Big\}\ud\Omega(\yb).
\end{align}
The fist time vanishes using Lemma \ref{lem:localization} for $\yb$. The second therm vanishes because the divergence annihilates the curl, after integrating by parts. The case $\Vb=\nabla_S Y_{\ell,m}$ is nearly identical. Using $\nabla_S Y_{\ell,m} = \xb\times\left(\nabla_S Y_{\ell,m}\times\xb\right)$ brings us to
\begin{align}
\DD_S^\delta\{\nabla_S Y_{\ell,m}\}(\xb) & = \int_{\Stwo} \nabla_S\cdot\left[\nabla_S Y_{\ell,m}(\yb)\right]\gamma_\delta(\xb\cdot\yb)\ud\Omega(\yb)\nonumber\\
& = \int_{\Stwo} \Delta_S Y_{\ell,m}(\yb)\gamma_\delta(\xb\cdot\yb)\ud\Omega(\yb).
\end{align}
We conclude with Lemma \ref{lem:localization} for $\Delta_S Y_{\ell,m}(\yb)$.

The proofs for $\CC_S^\delta$ are similar to those for $\DD_S^\delta$, so we skip them and continue with $\bs{\CC}_S^\delta$,
\begin{align}
\bs{\CC}_S^\delta\{Y_{\ell,m}\}(\xb) = \int_{\Stwo}\left[Y_{\ell,m}(\yb) - Y_{\ell,m}(\xb)\right]\xb\times\nabla_S^{\xb}\gamma_\delta(\xb\cdot\yb)\ud\Omega(\yb).
\label{eq:curlstepA}
\end{align}
Now, since $\xb\times\nabla_S^{\xb}\gamma_\delta(\xb\cdot\yb) = -\yb\times\nabla_S^{\yb}\gamma_\delta(\xb\cdot\yb)$, Eq.~\eqref{eq:curlstepA} becomes, after integrating by parts,
\begin{align}
\bs{\CC}_S^\delta\{Y_{\ell,m}\}(\xb) = \int_{\Stwo}\yb\times\nabla_S Y_{\ell,m}(\yb)\gamma_\delta(\xb\cdot\yb)\ud\Omega(\yb).
\end{align}
We conclude with Lemma \ref{lem:localization} for $\yb\times\Delta_S Y_{\ell,m}(\yb)$. To prove the formula for $\bs{\GG}^\delta_S$, we cross it with $\xb$.

We finish with $\bs{\CC}^\delta$. (The proofs for $(\bs{\CC}^\delta)^*$ are nearly identical.) We start with $\Vb=\nabla_S Y_{\ell,m}$. We get 
\begin{align}
\bs{\CC}_\delta\{\nabla_S Y_{\ell,m}\}(\xb) = & \; \int_{\Sph^2}\gamma'_\delta(\xb\cdot\yb)(\yb-\xb)\times\nabla_SY_{\ell,m}(\yb)\ud\Omega(\yb)\nonumber\\
& \; - \int_{\Sph^2}\gamma'_\delta(\xb\cdot\yb)(\yb-\xb)\times\nabla_SY_{\ell,m}(\xb)\ud\Omega(\yb),
\end{align}
and use Lemma \ref{lem:gen_localization} for $(\yb-\xb)\times\nabla_SY_{\ell,m}(\yb)$ and $\yb-\xb$. For $\Vb=\xb\times\nabla_S Y_{\ell,m}$, we obtain
\begin{align}
\bs{\CC}_\delta\{\xb\times\nabla_S Y_{\ell,m}\}(\xb) = & \int_{\Sph^2}\gamma'_\delta(\xb\cdot\yb)(\yb-\xb)\times(\yb\times\nabla_SY_{\ell,m}(\yb))\ud\Omega(\yb)\nonumber\\
& - \int_{\Sph^2}\gamma'_\delta(\xb\cdot\yb)(\yb-\xb)\times(\xb\times\nabla_SY_{\ell,m}(\xb))\ud\Omega(\yb),
\end{align}
and use Lemma \ref{lem:gen_localization} for $(\yb-\xb)\times(\yb\times\nabla_SY_{\ell,m}(\yb))$ and $\yb-\xb$. Finally, for $\Vb=Y_{\ell,m}\xb$, we arrive at
\begin{align}
\bs{\CC}_\delta\{\nabla_S Y_{\ell,m}\xb\}(\xb) = & \; \int_{\Sph^2}\gamma'_\delta(\xb\cdot\yb)(\yb-\xb)\times(Y_{\ell,m}(\yb)\yb)\ud\Omega(\yb)\nonumber\\
& \; - \int_{\Sph^2}\gamma'_\delta(\xb\cdot\yb)(\yb-\xb)\times(Y_{\ell,m}(\xb)\xb)\ud\Omega(\yb),
\end{align}
and use Lemma \ref{lem:gen_localization} for $(\yb-\xb)\times(Y_{\ell,m}(\yb)\yb)$ and $\yb-\xb$.
\end{proof}

We end this section by proving a nonlocal Stokes theorem. Let $\Sigma$ be a spherical patch and let $\partial\Sigma$ denote its boundary. Then the Stokes theorem states that
\begin{align}
\int_{\Sigma}\bs{\CC}^0\{\Vb\}\cdot\xb\ud\Omega = \int_{\Sigma}(\nabla\times\Vb)\cdot\xb\ud\Omega = \int_{\partial\Sigma}\Vb\cdot\bs{s}\ud\omega,
\end{align}
where $\bs{s}$ is a unit vector along $\partial\Sigma$. We have the following nonlocal version.

\begin{theorem}[Nonlocal Stokes theorem]\label{thm:nonlocalstokes}
For any smooth vector field $\Vb$, define the averaging operator $\bs{\mathcal{A}}^\delta$ via
\begin{align}
\bs{\mathcal{A}}^\delta\{\Vb\}(\xb) = \int_{\Sph^2}\gamma_\delta(\xb\cdot\yb)\Vb(\yb)\ud\Omega(\yb), \quad \xb\in\Stwo.
\end{align}
Then
\begin{align}
\int_{\Sigma}\bs{\CC}^\delta\{\Vb\}\cdot\xb\ud\Omega = \int_{\partial\Sigma}\bs{\mathcal{A}}^\delta\{\Vb\}\cdot\bs{s}\ud\omega.
\end{align}
\end{theorem}

\begin{proof}
Let 
\begin{align}
\Vb = \sum_{\ell=0}^{+\infty}\sum_{m=-\ell}^{+\ell}\left(V^s_{\ell,m}\nabla_S Y_{\ell,m} + V^t_{\ell,m}\xb\times\nabla_S Y_{\ell,m} + V^x_{\ell,m}Y_{\ell,m}\xb\right).
\end{align}
Then via Theorem \ref{thm:nonlocal_operators},
\begin{align}
\bs{\CC}^\delta\{\Vb\}\cdot\xb = -\sum_{\ell=0}^{+\infty}\Lambda_\ell^\delta\sum_{m=-\ell}^{+\ell}\ell(\ell+1)V^t_{\ell,m}Y_{\ell,m}.
\end{align}
On the other hand, using Lemma \ref{lem:localization}, we have $\bs{\mathcal{A}}^\delta\{\xb\times\nabla_S Y_{\ell,m}\} = \Lambda_\ell^\delta\xb\times\nabla_S Y_{\ell,m}$ and
\begin{align}
& \bs{\mathcal{A}}^\delta\{\nabla_S Y_{\ell,m}\} = \ell(\ell+1)\lambda_\ell\{\mu_\delta\}Y_{\ell,m}\xb + \lambda_\ell\{\mu_\delta + t\gamma_\delta\}\nabla_S Y_{\ell,m},\\
& \bs{\mathcal{A}}^\delta\{Y_{\ell,m}\xb\} = \lambda_\ell\{t\gamma_\delta\}Y_{\ell,m}\xb + \lambda_\ell\{\mu_\delta\}\nabla_S Y_{\ell,m}, 
\end{align}
where $\mu_\delta(t)=\int_0^t\gamma_\delta(s)ds$ is the primitive of $\gamma_\delta$ that vanishes at the origin. Therefore, via Theorem \ref{thm:local_operators},
\begin{align}
(\nabla\times\bs{\mathcal{A}}^\delta\{\Vb\})\cdot\xb = -\sum_{\ell=0}^{+\infty}\Lambda_\ell^\delta\sum_{m=-\ell}^{+\ell}\ell(\ell+1)V^t_{\ell,m}Y_{\ell,m} = \bs{\CC}^\delta\{\Vb\}\cdot\xb.
\end{align}
This yields
\begin{align}
\int_{\Sigma}\bs{\CC}^\delta\{\Vb\}\cdot\xb\ud\Omega = \int_{\Sigma}(\nabla\times\bs{\mathcal{A}}^\delta\{\Vb\})\cdot\xb\ud\Omega = \int_{\partial\Sigma}\bs{\mathcal{A}}^\delta\{\Vb\}\cdot\bs{s}\ud\omega,
\end{align}
by applying the local Stokes theorem to $\bs{\mathcal{A}}^\delta\{\Vb\}$.
\end{proof}

\section{Analysis of the scaling and convergence to the operators of local vector calculus}\label{sec:nonloc_anal}

We shall consider kernels $\bs{\alpha}_\delta(\xb,\yb)$ for which we recover from \eqref{eq:alpharho} the weakly singular kernels $\rho_\delta(\vert\xb-\yb\vert)$ introduced in Ref.~\citen{montanelli2018b} for the nonlocal operator \eqref{eq:nonloc_diff}. The latter were of the form
\begin{align*}
\rho_\delta(\abs{\xb-\yb}) = \dfrac{4(1+a)}{\pi\delta^{2+2a}}\frac{\rchi_{[0,\delta]}(\abs{\xb-\yb})}{\abs{\xb-\yb}^{2-2a}} = \dfrac{(1+a)2^{1+a}}{\pi\delta^{2+2a}(1-t)^{1-a}}\rchi_{[0,\delta]}(\sqrt{2(1-t)}),
\end{align*}
where $-1<a<1$, $0<\delta\leq2$, and $-1\le t=\xb\cdot\yb \le 1$. Using
\begin{align*}
2\bs{\alpha}_\delta(\xb,\yb)\cdot\bs{\alpha}_\delta(\xb,\yb) = 2\left[1-\xb\cdot\yb\right]\left[1+\xb\cdot\yb\right]\left[\gamma_\delta'(\xb\cdot\yb)\right]^2,
\end{align*}
the condition $\rho_\delta(\vert\xb-\yb\vert)=2\vert\bs{\alpha}_\delta(\xb,\yb)\vert^2$ of Theorem \ref{thm:nonloc_composition} yields a non-negative
\begin{align*}
\gamma_\delta'(t) = \sqrt{\frac{1+a}{\pi}}\frac{2^{\frac{a}{2}}}{\delta^{1+a}}\frac{1}{(1-t)^{\frac{2-a}{2}}}\frac{1}{\sqrt{1+t}}\rchi_{[0,\delta]}(\sqrt{2(1-t)}).
\end{align*}
We note that $\sqrt{1-t^2}\gamma_\delta'\in L^1(-1,1)$. Integrating gives a monotonically increasing
\begin{align}
\gamma_\delta(t) = & \; \Bigg\{\frac{1}{2\pi}\frac{2}{\delta^2} + \frac{1}{a}\sqrt{\frac{2(1+a)}{\pi\delta^2}}\Bigg[\frac{2}{2+a}\pFq{2}{1}\left(\begin{array}{c} \frac{1}{2}, \frac{a}{2}\\ \frac{4+a}{2}\end{array}; \frac{\delta^2}{4}\right)\nonumber\\
& \; - \left(\frac{2(1-t)}{\delta^2}\right)^{\frac{a}{2}}\pFq{2}{1}\left(\begin{array}{c}\frac{1}{2}, \frac{a}{2}\\ \frac{2+a}{2}\end{array}; \frac{1-t}{2}\right)\Bigg]\Bigg\}\rchi_{[0,\delta]}(\sqrt{2(1-t)}),\label{eq:gammadelta}
\end{align}
on $(-1,1)$, $a\ne0$, with the constant chosen such that
\begin{align*}
\Lambda_0^\delta = 2\pi \int_{-1}^1\gamma_\delta(t)\ud t = 1.
\end{align*}
In the case that $a=0$, we interpret the above formula for $\gamma_\delta$ in the limiting sense.

\begin{lemma}[Positivity of the kernel]\label{lemma:GammaIsPositive}
For any $-1<a<1$, if $0<\delta\le\frac{1}{4}$, then $\gamma_\delta(t)>0$ on $[1-\delta^2/2,1]$.
\end{lemma}

\begin{proof}
Since $\gamma_\delta(t)$ is monotonically increasing, it is sufficient to check the condition at the left endpoint $t=1-\delta^2/2$. By term-by-term inequalities on the hypergeometric series,
\begin{align*}
\pFq{2}{1}\left(\begin{array}{c}\frac{1}{2}, \frac{2+a}{2}\\ \frac{4+a}{2}\end{array}; \frac{\delta^2}{4}\right) < \pFq{2}{1}\left(\begin{array}{c}\frac{1}{2}, \frac{4+a}{2}\\ \frac{4+a}{2}\end{array}; \frac{\delta^2}{4}\right) = \frac{1}{\sqrt{1-\delta^2/4}}.
\end{align*}
Then
\begin{align*}
\gamma_\delta(1-\delta^2/2) & > \frac{1}{2\pi}\frac{2}{\delta^2} - \frac{1}{2+a}\sqrt{\frac{2(1+a)}{\pi\delta^2}}\frac{1}{\sqrt{1-\delta^2/4}},\\
& > \frac{1}{2\pi}\frac{2}{\delta^2} - \sqrt{\frac{4}{\pi\delta^2(1-\delta^2/4)}}.
\end{align*}
The right-hand side is also monotonic as $\delta\to0^+$; suffice it to prove that $\gamma_{\frac{1}{4}}(\frac{31}{32}) > 0$. For any $0<\delta \le \frac{1}{4}$,
\begin{align*}
\gamma_\delta(1-\delta^2/2) \ge \gamma_{\frac{1}{4}}(\tfrac{31}{32}) = 16\left(\frac{1}{\pi}-\sqrt{\frac{16}{63\pi}}\right) > 0.
\end{align*}
\end{proof}

\begin{lemma}[Limiting values]
For every $\ell\ge1$, $\lim_{\delta\to0}\Lambda_\ell^\delta = 1$ and for every $\ell\ge0$, $\lim_{\delta\to0}\Theta_\ell^\delta = 1$. Specifically, for all $\epsilon>0$ and $\ell\ge1$,
\begin{align*}
\delta\leq\min(1/4,\sqrt{4\epsilon/[\ell(\ell+1)]}) \quad \Longrightarrow \quad \vert\Lambda_\ell^\delta - 1\vert \leq \epsilon.
\end{align*}
Similarly, for all $\epsilon>0$ and $\ell\ge0$,
\begin{align*}
\delta\leq\min(1/4,\epsilon/[\ell(\ell+1) + 1]) \quad \Longrightarrow \quad \vert\Theta_\ell^\delta - 1\vert \leq \epsilon.
\end{align*}
\end{lemma}

\begin{proof}
Since
\begin{align}
\Lambda_\ell^\delta - 1 = 2\pi\int_{1-\delta^2/2}^1 \left[P_\ell(t)-1\right]\gamma_\delta(t)\ud t,
\end{align}
we use the Lagrange remainder theorem
\begin{align}
P_\ell(t) = P_\ell(1) + P_\ell'(\xi)(t-1), \quad \xi\in(1-\delta^2/2,1),
\end{align}
and the fact that $\max_{-1\le t\le 1}\vert P_\ell'(t)\vert\leq\ell(\ell+1)/2$. Thus
\begin{align}
\vert\Lambda_\ell^\delta - 1\vert \leq 2\pi\int_{1-\delta^2/2}^1\vert P_\ell'(t)\vert\vert t-1\vert\vert\gamma_\delta(t)\vert\ud t \leq \frac{\pi\ell(\ell+1)\delta^2}{2}\int_{1-\delta^2/2}^1 \vert\gamma_\delta(t)\vert\ud t.
\end{align}
For $\delta\leq1/4$, the integrand is nonnegative and integrates to $1/(2\pi)$. This yields
\begin{align}
\vert\Lambda_\ell^\delta - 1\vert \leq \frac{\ell(\ell+1)\delta^2}{4}.
\end{align}

For the second limit, we integrate by parts
\begin{align}
\Theta_\ell^\delta = 2\pi[P_\ell(t)(1-t)\gamma_\delta(t)]_{1-\delta^2/2}^1 - 2\pi\int_{1-\delta^2/2}^1P_\ell'(t)(1-t)\gamma_\delta(t)\ud t + \Lambda_\ell^\delta,
\end{align}
and note that the first two terms go to $0$ and the last one goes to $1$. Specifically, we have
\begin{align}
\vert\Theta_\ell^\delta - 1\vert \leq & \; 2\pi\abs{[P_\ell(t)(1-t)\gamma_\delta(t)]_{1-\delta^2/2}^1} + 2\pi\int_{1-\delta^2/2}^1\vert P_\ell'(t)\vert\vert 1-t\vert\vert\gamma_\delta(t)\vert\ud t\nonumber\\
& \; + \vert\Lambda_\ell^\delta - 1\vert.
\end{align}
This immediately gives, for $\delta\leq1/4$,
\begin{align}
\vert\Theta_\ell^\delta - 1\vert \leq 2\pi\abs{[P_\ell(t)(1-t)\gamma_\delta(t)]_{1-\delta^2/2}^1} + \frac{\ell(\ell+1)\delta^2}{2}.
\end{align}
Furthermore
\begin{align}
& \; 2\pi[P_\ell(t)(1-t)\gamma_\delta(t)]_{1-\delta^2/2}^1\nonumber\\
= & \; 2\pi\lim_{\epsilon\to0}\left\{P_\ell(1-\epsilon)\epsilon \gamma_\delta(1-\epsilon) - P_\ell(1-\delta^2/2)\delta^2/2\gamma_\delta(1-\delta^2/2)\right\},
\end{align}
with
\begin{align}
& \epsilon\gamma_\delta(1-\epsilon) = \frac{1}{2\pi} - \frac{\epsilon}{2+a}\sqrt{\frac{(1+a)}{\pi\epsilon}}\pFq{2}{1}\left(\begin{array}{c}\frac{1}{2}, \frac{2+a}{2}\\ \frac{4+a}{2}\end{array}; \frac{\epsilon}{2}\right), \\
& \delta^2/2\gamma_\delta(1-\delta^2/2) = \frac{1}{2\pi} - \frac{\delta^2}{2(2+a)}\sqrt{\frac{2(1+a)}{\pi\delta^2}}\pFq{2}{1}\left(\begin{array}{c}\frac{1}{2}, \frac{2+a}{2}\\ \frac{4+a}{2}\end{array}; \frac{\delta^2}{4}\right).
\end{align}
Therefore,
\begin{align}
& \; 2\pi[P_\ell(t)(1-t)\gamma_\delta(t)]_{1-\delta^2/2}^1\nonumber\\
= & \; 1 - P_\ell(1-\delta^2/2) + \frac{\pi\delta^2}{2+a}\sqrt{\frac{2(1+a)}{\pi\delta^2}}\pFq{2}{1}\left(\begin{array}{c}\frac{1}{2}, \frac{2+a}{2}\\ \frac{4+a}{2}\end{array}; \frac{\delta^2}{4}\right).
\end{align}
Finally, we note that 
\begin{align}
\vert 1 - P_\ell(1-\delta^2/2)\vert \leq \frac{\ell(\ell+1)\delta^2}{4},
\end{align}
as well as
\begin{align}
\frac{\pi\delta^2}{2+a}\sqrt{\frac{2(1+a)}{\pi\delta^2}}\abs{\pFq{2}{1}\left(\begin{array}{c}\frac{1}{2}, \frac{2+a}{2}\\ \frac{4+a}{2}\end{array}; \frac{\delta^2}{4}\right)} 
& \leq \frac{\delta}{\sqrt{\pi}}\abs{\pFq{2}{1}\left(\begin{array}{c}\frac{1}{2}, \frac{2+a}{2}\\ \frac{4+a}{2}\end{array}; \frac{\delta^2}{4}\right)},\nonumber\\
& \leq \frac{\delta}{\sqrt{\pi}}\frac{1}{\sqrt{1-\delta^2/4}} \leq \delta.
\end{align}
This yields
\begin{align}
\vert\Theta_\ell^\delta - 1\vert \leq \frac{3\ell(\ell+1)\delta^2}{4} + \delta \leq \left[\ell(\ell+1) + 1\right]\delta.
\end{align}
\end{proof}

We note that the previous lemma implies that for all $\epsilon>0$ and $\ell\ge0$,
\begin{align}
\delta\leq\min\left(1/4,\epsilon/[3\ell(\ell+1) + 3]\right) \quad \Longrightarrow \quad \vert\pm\Theta_0^\delta + \Theta_\ell^\delta - \Lambda_\ell^\delta \mp 1\vert \leq \epsilon.
\end{align}

Let $u$ be a scalar field on $\Stwo$ expanded in spherical harmonics as in \eqref{eq:scalar_spherical_expansion}, and let $\Vb$ be a vector field on $\Stwo$ expanded in vector spherical harmonics as in \eqref{eq:vector_spherical_expansion}. Owing to the orthogonality conditions
\begin{align}
& \int_{\Stwo}\conj{Y_{\ell',m'}(\xb)}Y_{\ell,m}(\xb)\ud\Omega(\xb) = \delta_{\ell,\ell'}\delta_{m,m'},\\
& \int_{\Stwo}\conj{\nabla_S Y_{\ell',m'}(\xb)}\cdot\nabla_S Y_{\ell,m}(\xb)\ud\Omega(\xb) = \ell(\ell+1)\delta_{\ell,\ell'}\delta_{m,m'},\\
& \int_{\Stwo}\conj{\xb\times\nabla_S Y_{\ell',m'}(\xb)}\cdot\xb\times\nabla_S Y_{\ell,m}(\xb)\ud\Omega(\xb) = \ell(\ell+1)\delta_{\ell,\ell'}\delta_{m,m'},\\
& \int_{\Stwo}\conj{Y_{\ell',m'}(\xb)\xb}\cdot Y_{\ell,m}(\xb)\xb\ud\Omega(\xb) = \delta_{\ell,\ell'}\delta_{m,m'},
\end{align}
the following definitions of Sobolev spaces are natural.\cite{atkinson2012}

\begin{definition}[Sobolev spaces]
Let $s\in\R$. The space $H^s(\Stwo,\R)$ is the completion of $C^\infty(\Stwo,\R)$ with respect to the norm
\begin{align}
\Vert u\Vert_{H^s(\Stwo,\R)}^2 = \sum_{\ell=0}^{+\infty}\sum_{m=-\ell}^{+\ell}\left(\ell+\tfrac{1}{2}\right)^{2s}\abs{u_{\ell,m}}^2.
\end{align}
The space $H^s(\Stwo,T\Stwo)$ is the completion of $C^\infty(\Stwo,T\Stwo)$ with respect to the norm
\begin{align}
\Vert\Vb\Vert_{H^s(\Stwo,T\Stwo)}^2 = \sum_{\ell=0}^{+\infty}\sum_{m=-\ell}^{+\ell}\left(\ell+\tfrac{1}{2}\right)^{2s+2}\left\{\abs{V_{\ell,m}^s}^2 + \abs{V_{\ell,m}^t}^2\right\}.
\end{align}
The space $H^s(\Stwo,\R^3)$ is the completion of $C^\infty(\Stwo,\R^3)$ with respect to the norm
\begin{align}
\Vert\Vb\Vert_{H^s(\Stwo,\R^3)}^2 = & \; \sum_{\ell=0}^{+\infty}\sum_{m=-\ell}^{+\ell}\Big\{\left(\ell+\tfrac{1}{2}\right)^{2s+2}\abs{V_{\ell,m}^s}^2 + \left(\ell+\tfrac{1}{2}\right)^{2s+2}\abs{V_{\ell,m}^t}^2\nonumber\\
& \; + \left(\ell+\tfrac{1}{2}\right)^{2s}\abs{V_{\ell,m}^x}^2\Big\}.
\end{align}
\end{definition}

For any $s\geq0$, consider
\begin{align*}
& \DD_S^0 : \Vb\in H^s(\Stwo,\R^3) \mapsto \DD_S^0\{\Vb\}\in H^{s-1}(\Stwo,\R),\\
& \CC_S^0 : \Vb\in H^s(\Stwo,\R^3) \mapsto \CC_S^0\{\Vb\}\in H^{s-1}(\Stwo,\R),\\
& \bs{\GG}_S^0 : u\in H^s(\Stwo,\R) \mapsto \bs{\GG}_S^0\{u\}\in H^{s-1}(\Stwo,T\Stwo),\\
& \bs{\CC}_S^0 : u\in H^s(\Stwo,\R) \mapsto \bs{\CC}_S^0\{u\}\in H^{s-1}(\Stwo,T\Stwo),\\
& \bs{\CC}^0: \Vb\in H^s(\Stwo,\R^3) \mapsto \bs{\CC}^0\{\Vb\}\in H^{s-1}(\Stwo,\R^3),\\
& (\bs{\CC}^0)^* : \Vb\in H^s(\Stwo,\R^3) \mapsto (\bs{\CC}^0)^*\{\Vb\}\in H^{s-1}(\Stwo,\R^3).
\end{align*}

\begin{lemma}[Boundedness of local operators] 
For every $s\geq0$, the local operators are bounded and
\begin{align}
& \Vert\DD_S^0\Vert_s = \sup_{\bs{0}\neq \Vb\in H^s(\Stwo,\R^3)}\frac{\Vert\DD_S^0\{\Vb\}\Vert_{H^{s-1}(\Stwo,\R)}}{\Vert\Vb\Vert_{H^s(\Stwo,\R^3)}} \leq 1, && \Vert\CC_S^0\Vert_s \leq 1, \label{eq:bounded-loc-div} \\
& \Vert\bs{\GG}_S^0\Vert_s = \sup_{0\neq u\in H^s(\Stwo,\R)}\frac{\Vert\bs{\GG}_S^0\{u\}\Vert_{H^{s-1}(\Stwo,T\Stwo)}}{\Vert u\Vert_{H^s(\Stwo,\R)}} = 1, && \Vert\bs{\CC}_S^0\Vert_s = 1, \label{eq:bounded-loc-grad} \\
& \Vert\bs{\CC}^0\Vert_s = \sup_{\bs{0}\neq \Vb\in H^s(\Stwo,\R^3)}\frac{\Vert\bs{\CC}^0\{\Vb\}\Vert_{H^{s-1}(\Stwo,\R^3)}}{\Vert\Vb\Vert_{H^s(\Stwo,\R^3)}} \leq \sqrt{2}, && \Vert(\bs{\CC}^0)^*\Vert_s \leq \sqrt{2}. \label{eq:bounded-loc-curl}
\end{align}
\end{lemma}

\begin{proof}
We start with \eqref{eq:bounded-loc-div}. Let $s\geq0$ and consider
\begin{align*}
\Vb = \sum_{\ell=0}^{+\infty}\sum_{m=-\ell}^{+\ell}\Big\{V^s_{\ell,m}\nabla_S Y_{\ell,m} + V^t_{\ell,m}\xb\times\nabla_S Y_{\ell,m} + V^x_{\ell,m}Y_{\ell,m}\xb\Big\} \in H^s(\Stwo,\R^3).
\end{align*}
Then
\begin{align}
\DD_S^0\{\Vb\} = \sum_{\ell=0}^{+\infty}\sum_{m=-\ell}^{+\ell}\left[-\ell(\ell+1)\right]V^s_{\ell,m}Y_{\ell,m}.
\end{align}
This yields
\begin{align}
\Vert\DD_S^0\{\Vb\}\Vert_{H^{s-1}(\Stwo,\R)}^2 = \sum_{\ell=0}^{+\infty}\sum_{m=-\ell}^{+\ell}\left(\ell+\tfrac{1}{2}\right)^{2s-2}\abs{\ell(\ell+1)}^2\abs{V^s_{\ell,m}}^2.
\end{align}
Since $\left(\ell+\tfrac{1}{2}\right)^{2s-2}\abs{\ell(\ell+1)}^2\leq\left(\ell+\tfrac{1}{2}\right)^{2s+2}$ for all $\ell\geq0$,
\begin{align}
\Vert\DD_S^0\{\Vb\}\Vert_{H^{s-1}(\Stwo,\R)}^2 \leq \sum_{\ell=0}^{+\infty}\sum_{m=-\ell}^{+\ell}\left(\ell+\tfrac{1}{2}\right)^{2s+2}\abs{V^s_{\ell,m}}^2 \leq \Vert\Vb\Vert_{H^s(\Stwo,\R^3)}^2.
\end{align}
The proof is similar for $\CC_S^0$ since
\begin{align}
\CC_S^0\{\Vb\} = \sum_{\ell=0}^{+\infty}\sum_{m=-\ell}^{+\ell}\left[\ell(\ell+1)\right]V^t_{\ell,m}Y_{\ell,m}.
\end{align}

We continue with \eqref{eq:bounded-loc-grad}. Again, let $s\geq0$. We note that, for $u\in H^s(\Stwo,\R)$,
\begin{align}
u = \sum_{\ell=0}^{+\infty}\sum_{m=-\ell}^{+\ell} u_{\ell,m} Y_{\ell,m} \quad \Longrightarrow \quad \bs{\GG}_S^0\{u\} = \sum_{\ell=0}^{+\infty}\sum_{m=-\ell}^{+\ell} u_{\ell,m} \nabla_S Y_{\ell,m}.
\end{align}
This yields
\begin{align}
\Vert\bs{\GG}_S^0\{u\}\Vert_{H^{s-1}(\Stwo,T\Stwo)}^2 = \sum_{\ell=0}^{+\infty}\sum_{m=-\ell}^{+\ell}\left(\ell+\tfrac{1}{2}\right)^{2s}\abs{u_{\ell,m}}^2 = \Vert u\Vert_{H^s(\Stwo,\R)}^2.
\end{align}
The proof is similar for $\bs{\CC}_S^0$ since
\begin{align}
\bs{\CC}_S^0\{u\} = \sum_{\ell=0}^{+\infty}\sum_{m=-\ell}^{+\ell} u_{\ell,m} \xb\times\nabla_S Y_{\ell,m}.
\end{align}

We finish with \eqref{eq:bounded-loc-curl}. Let $s\geq0$ and
\begin{align*}
\Vb = \sum_{\ell=0}^{+\infty}\sum_{m=-\ell}^{+\ell}\Big\{V^s_{\ell,m}\nabla_S Y_{\ell,m} + V^t_{\ell,m}\xb\times\nabla_S Y_{\ell,m} + V^x_{\ell,m}Y_{\ell,m}\xb\Big\} \in H^s(\Stwo,\R^3).
\end{align*}
Then
\begin{align}
\bs{\CC}_S^0\{\Vb\} = & \; \sum_{\ell=0}^{+\infty}\sum_{m=-\ell}^{+\ell}\Big\{-V^t_{\ell,m}\nabla_S Y_{\ell,m} + (V^s_{\ell,m}-V^x_{\ell,m})\xb\times\nabla_S Y_{\ell,m}\nonumber\\
& \; - \ell(\ell+1)V^t_{\ell,m}Y_{\ell,m}\xb\Big\}.
\end{align}
This yields
\begin{align*}
& \;\Vert\bs{\CC}^0\{\Vb\}\Vert^2_{H^{s-1}(\Stwo,\R^3)} \\
= & \; \sum_{\ell=0}^{+\infty}\sum_{m=-\ell}^{+\ell}\Big\{\left(\ell+\tfrac{1}{2}\right)^{2s}\abs{V_{\ell,m}^t}^2 + \left(\ell+\tfrac{1}{2}\right)^{2s}\abs{V^s_{\ell,m}-V^x_{\ell,m}}^2\\
& \; + \left(\ell+\tfrac{1}{2}\right)^{2s-2}\abs{\ell(\ell+1)}^2\abs{V_{\ell,m}^t}^2\Big\},\\
\leq & \; 2\sum_{\ell=0}^{+\infty}\sum_{m=-\ell}^{+\ell}\left\{\left(\ell+\tfrac{1}{2}\right)^{2s+2}\abs{V^s_{\ell,m}}^2 + \left(\ell+\tfrac{1}{2}\right)^{2s+2}\abs{V_{\ell,m}^t}^2 + \left(\ell+\tfrac{1}{2}\right)^{2s}\abs{V^x_{\ell,m}}^2\right\},\\
= & \; 2\Vert\Vb\Vert^2_{H^s(\Stwo,\R^3)}.
\end{align*}
The proof is similar for $(\bs{\CC}_S^0)^*$ since
\begin{align}
(\bs{\CC}_S^0)^*\{\Vb\} = & \; \sum_{\ell=0}^{+\infty}\sum_{m=-\ell}^{+\ell}\Big\{V^t_{\ell,m}\nabla_S Y_{\ell,m} + (-V^s_{\ell,m}-V^x_{\ell,m})\xb\times\nabla_S Y_{\ell,m}\nonumber\\
& \; - \ell(\ell+1)V^t_{\ell,m}Y_{\ell,m}\xb\Big\}.
\end{align}
\end{proof}

For any $s\geq0$ and $0<\delta\leq2$, consider now the nonlocal operators
\begin{align*}
& \DD_S^\delta : \Vb\in H^s(\Stwo,\R^3)\mapsto \DD_S^\delta\{\Vb\}\in H^{s-1}(\Stwo,\R),\\
& \CC_S^\delta : \Vb\in H^s(\Stwo,\R^3)\mapsto \CC_S^\delta\{\Vb\}\in H^{s-1}(\Stwo,\R),\\
& \bs{\GG}_S^\delta : u\in H^s(\Stwo,\R)\mapsto \bs{\GG}_S^\delta\{u\}\in H^{s-1}(\Stwo,T\Stwo),\\
& \bs{\CC}_S^\delta : u\in H^s(\Stwo,\R)\mapsto \bs{\CC}_S^\delta\{u\}\in H^{s-1}(\Stwo,T\Stwo),\\
& \bs{\CC}^\delta: \Vb\in H^s(\Stwo,\R^3)\mapsto \bs{\CC}^\delta\{\Vb\}\in H^{s-1}(\Stwo,\R^3),\\
& (\bs{\CC}^\delta)^* :  \Vb\in H^s(\Stwo,\R^3)\mapsto (\bs{\CC}^\delta)^*\{\Vb\}\in H^{s-1}(\Stwo,\R^3).
\end{align*}

\begin{lemma}[Boundedness of nonlocal operators]\label{lem:bounded}
For every $s\geq0$ and horizon $0<\delta\leq2$, the nonlocal operators are bounded.
\end{lemma}

\begin{proof}
We start with $\DD_S^\delta$. Let $s\geq0$ and $0<\delta\leq2$. Consider
\begin{align*}
\Vb = \sum_{\ell=0}^{+\infty}\sum_{m=-\ell}^{+\ell}\Big\{V^s_{\ell,m}\nabla_S Y_{\ell,m} + V^t_{\ell,m}\xb\times\nabla_S Y_{\ell,m} + V^x_{\ell,m}Y_{\ell,m}\xb\Big\} \in H^s(\Stwo,\R^3).
\end{align*}
Then
\begin{align}
\DD_S^\delta\{\Vb\} = \sum_{\ell=0}^{+\infty}\sum_{m=-\ell}^{+\ell}\left[-\Lambda_\ell^\delta\ell(\ell+1)\right]V^s_{\ell,m}Y_{\ell,m}.
\end{align}
This yields
\begin{align}
\Vert\DD_S^\delta\{\Vb\}\Vert_{H^{s-1}(\Stwo,\R)}^2 = \sum_{\ell=0}^{+\infty}\sum_{m=-\ell}^{+\ell}\left(\ell+\tfrac{1}{2}\right)^{2s-2}\abs{\Lambda_\ell^\delta}^2\abs{\ell(\ell+1)}^2\abs{V^s_{\ell,m}}^2.
\end{align}
Since there exists $c>0$ such that (recall that $P_\ell(t)\leq 1$ for all $\ell\geq0$ and $t\in[-1,1]$)
\begin{align}
\abs{\Lambda_\ell^\delta} \leq 2\pi\int_{-1}^1\abs{P_\ell(t)}\abs{\gamma_\delta(t)}\ud t\leq 2\pi\int_{-1}^1\abs{\gamma_\delta(t)}\ud t\leq c,
\end{align}
uniformly in $\delta$ and $\ell$, we have
\begin{align}
\Vert\DD_S^\delta\{\Vb\}\Vert_{H^{s-1}(\Stwo,\R)}^2 & \leq c^2\sum_{\ell=0}^{+\infty}\sum_{m=-\ell}^{+\ell}\left(\ell+\tfrac{1}{2}\right)^{2s-2}\abs{\ell(\ell+1)}^2\abs{V^s_{\ell,m}}^2,\nonumber\\
& = c^2\Vert\DD_S^0\{\Vb\}\Vert_{H^{s-1}(\Stwo,\R)}^2.
\end{align}
(Note that $c=1$ for $\delta\leq1/4$ since in that case $\gamma_\delta$ is nonnegative.) The proof is similar for $\CC_S^0$ since
\begin{align}
\CC_S^\delta\{\Vb\} = \sum_{\ell=0}^{+\infty}\sum_{m=-\ell}^{+\ell}\left[\Lambda_\ell^\delta\ell(\ell+1)\right]V^t_{\ell,m}Y_{\ell,m}.
\end{align}

We continue with $\bs{\GG}_S^\delta$. Let $s\geq0$ and $0<\delta\leq2$. We note that, for $u\in H^s(\Stwo,\R)$,
\begin{align}
u = \sum_{\ell=0}^{+\infty}\sum_{m=-\ell}^{+\ell} u_{\ell,m} Y_{\ell,m} \quad \Longrightarrow \quad \bs{\GG}_S^\delta\{u\} = \sum_{\ell=0}^{+\infty}\sum_{m=-\ell}^{+\ell} \Lambda_\ell^\delta u_{\ell,m} \nabla_S Y_{\ell,m}.
\end{align}
This yields, with the same constant $c>0$ as before,
\begin{align}
\Vert\bs{\GG}_S^\delta\{u\}\Vert_{H^{s-1}(\Stwo,T\Stwo)}^2 & = \sum_{\ell=0}^{+\infty}\sum_{m=-\ell}^{+\ell}\left(\ell+\tfrac{1}{2}\right)^{2s}\abs{\Lambda_\ell^\delta}^2\abs{u_{\ell,m}}^2,\nonumber\\
& \leq c^2\Vert\bs{\GG}_S^0\{u\}\Vert_{H^{s-1}(\Stwo,T\Stwo)}^2.
\end{align}
The proof is similar for $\bs{\CC}_S^\delta$ since
\begin{align}
\bs{\CC}_S^\delta\{u\} = \sum_{\ell=0}^{+\infty}\sum_{m=-\ell}^{+\ell} \Lambda_\ell^\delta u_{\ell,m} \xb\times\nabla_S Y_{\ell,m}.
\end{align}

We finish with $\bs{\CC}_S^\delta$. Let $s\geq0$, and $0<\delta\leq2$, and consider
\begin{align*}
\Vb = \sum_{\ell=0}^{+\infty}\sum_{m=-\ell}^{+\ell}\left\{V^s_{\ell,m}\nabla_S Y_{\ell,m} + V^t_{\ell,m}\xb\times\nabla_S Y_{\ell,m} + V^x_{\ell,m}Y_{\ell,m}\xb\right\} \in H^s(\Stwo,\R^3).
\end{align*}
Then
\begin{align*}
\bs{\CC}_S^\delta\{\Vb\} = & \sum_{\ell=0}^{+\infty}\sum_{m=-\ell}^{+\ell}\Big\{\left[-\Theta_0^\delta + \Theta_\ell^\delta - \Lambda_\ell^\delta\right]V^t_{\ell,m}\nabla_S Y_{\ell,m}\nonumber\\
& + \left(\left[\Theta_0^\delta + \Theta_\ell^\delta - \Lambda_\ell^\delta\right]V^s_{\ell,m} - \Lambda_\ell^\delta V^x_{\ell,m}\right)\xb\times\nabla_S Y_{\ell,m}- \Lambda_\ell^\delta\ell(\ell+1)V^t_{\ell,m}Y_{\ell,m}\xb\Big\}.
\end{align*}
Since there exists $b>0$ such that
\begin{align}
\abs{\pm\Theta_0^\delta + \Theta_\ell^\delta - \Lambda_\ell^\delta} \leq \abs{\Theta_0^\delta} + \abs{\Theta_\ell^\delta} + \abs{\Lambda_\ell^\delta} \leq b, \quad \abs{\Lambda_\ell^\delta} \leq c < b,
\end{align}
uniformly in $\delta$ and $\ell$, we obtain
\begin{align}
\Vert\bs{\CC}^\delta\{\Vb\}\Vert_{H^{s-1}(\Stwo,\R^3)}^2 \leq 2b^2\Vert\Vb\Vert^2_{H^s(\Stwo,\R^3)}.
\end{align}
The proof is similar for $(\bs{\CC}_S^\delta)^*$.
\end{proof}

Similarly, we can show that the averaging operator $\bs{\mathcal{A}}^\delta$ from Theorem \ref{thm:nonlocalstokes} is bounded from $H^s(\Stwo,\R^3)$ to $H^{s-1}(\Stwo,\R^3)$ for any $s\geq0$. Note that Lemma \ref{lem:bounded} extends the results of the previous sections---ranging from the definition of nonlocal operators (Definition \ref{def:nonloc_operators}) to the nonlocal Stokes theorem (Theorem \ref{thm:nonlocalstokes})---to more general function spaces.

We now turn to the strong convergence of the weighted nonlocal operators to their local counterparts.

\begin{theorem}[Strong convergence to local operators]
For every $s\geq0$, in the limit as $\delta\to0$, the weighted nonlocal operators converge strongly to the local analogues. That is,
\begin{align}
& \lim_{\delta\to0}\Vert\DD_S^\delta\{\Vb\} - \DD_S^0\{\Vb\}\Vert_{H^{s-1}(\Stwo,\R)} = 0 && \forall \Vb\in H^s(\Stwo,\R^3),\\
& \lim_{\delta\to0}\Vert\CC_S^\delta\{\Vb\} - \CC_S^0\{\Vb\}\Vert_{H^{s-1}(\Stwo,\R)} = 0 && \forall \Vb\in H^s(\Stwo,\R^3),\\
& \lim_{\delta\to0}\Vert\bs{\GG}_S^\delta\{u\} - \bs{\GG}_0^\delta\{u\}\Vert_{H^{s-1}(\Stwo,T\Stwo)} = 0 && \forall u\in H^s(\Stwo,\R),\\
& \lim_{\delta\to0}\Vert\bs{\CC}_S^\delta\{u\} - \bs{\CC}_S^0\{u\}\Vert_{H^{s-1}(\Stwo,T\Stwo)} = 0 && \forall u\in H^s(\Stwo,\R),\\
& \lim_{\delta\to0}\Vert\bs{\CC}^\delta\{\Vb\} - \bs{\CC}^0\{\Vb\}\Vert_{H^{s-1}(\Stwo,\R^3)} = 0 && \forall \Vb\in H^s(\Stwo,\R^3),\\
& \lim_{\delta\to0}\Vert(\bs{\CC}^\delta)^*\{\Vb\} - (\bs{\CC}^0)^*\{\Vb\}\Vert_{H^{s-1}(\Stwo,\R^3)} = 0 && \forall \Vb\in H^s(\Stwo,\R^3).
\end{align}
\end{theorem}

\begin{proof}
The six cases are similar; we only prove the statements for $\DD_S^\delta$ and $\bs{\CC}^\delta$. Let $s\geq0$, $\Vb\in H^s(\Stwo,\R^3)$, $\varepsilon>0$, $n\geq1$, and
\begin{align}
\Vb_n = \sum_{\ell=0}^{+n}\sum_{m=-\ell}^{+\ell}\left(V^s_{\ell,m}\nabla_S Y_{\ell,m}(\xb) + V^t_{\ell,m}\xb\times\nabla_S Y_{\ell,m}(\xb) + V^x_{\ell,m}Y_{\ell,m}(\xb)\xb\right),
\end{align}
be the degree-$n$ truncation of $\Vb$. There exists $\delta\leq\min(1/4,\sqrt{4\epsilon/[n(n+1)]})$ such that
\begin{align}
& \; \norm{(\DD_S^\delta - \DD_S^0)\{\Vb\}}_{H^{s-1}(\Stwo,\R)}\nonumber\\
\leq & \; \norm{(\DD_S^\delta - \DD_S^0)\{\Vb_n\}}_{H^{s-1}(\Stwo,\R)} + (\norm{\DD_S^\delta}_s + \norm{\DD_S^0}_s)\norm{\Vb-\Vb_n}_{H^s(\Stwo,\R^3)},\nonumber\\
\leq & \; \epsilon\norm{\Vb}_{H^s(\Stwo,\R^3)} + 2\norm{\Vb-\Vb_n}_{H^s(\Stwo,\R^3)},
\end{align}
since
\begin{align}
\norm{(\DD_S^\delta - \DD_S^0)\{\Vb_n\}}_{H^{s-1}(\Stwo,\R)}^2 & = \sum_{\ell=0}^{+n}\sum_{m=-\ell}^{+\ell}\left(\ell+\tfrac{1}{2}\right)^{2s-2}\abs{\Lambda_\ell^\delta - 1}^2\abs{\ell(\ell+1)}^2\abs{V^s_{\ell,m}}^2,\nonumber\\
& \leq \epsilon^2\norm{\Vb}_{H^s(\Stwo,\R^3)}^2.
\end{align}
Thus, the norm is bounded by the first term that is arbitrarily small and the second term that converges to $0$ as $n\to\infty$ or equivalently as $\delta\to0$; for definiteness, choose
\begin{equation*}
\varepsilon = \frac{1}{n(n+1)},\quad \delta = \sqrt{\frac{4\epsilon}{n(n+1)}} = \frac{2}{n(n+1)},\quad\hbox{and take}\quad n\to\infty.
\end{equation*}
Convergence of the second term follows from the completeness of vector spherical harmonics in $H^s(\Stwo,\R^3)\hookrightarrow L^2(\Stwo,\R^3)$.

Similarly, there exists $\delta\leq\min\left(1/4,\epsilon/[3\ell(\ell+1) + 3]\right)$ such that
\begin{align}
& \; \Vert(\bs{\CC}^\delta - \bs{\CC}^0)\{\Vb\}\Vert_{H^{s-1}(\Stwo,\R^3)}\nonumber\\
\leq & \; \Vert(\bs{\CC}^\delta - \bs{\CC}^0)\{\Vb_n\}\Vert_{H^{s-1}(\Stwo,\R^3)} + (\Vert\bs{\CC}^\delta\Vert_s + \Vert\bs{\CC}^0\Vert_s)\norm{\Vb-\Vb_n}_{H^s(\Stwo,\R^3)},\nonumber\\
\leq & \; \sqrt{2}\epsilon\norm{\Vb}_{H^s(\Stwo,\R^3)} + \sqrt{2}(1+b)\norm{\Vb-\Vb_n}_{H^s(\Stwo,\R^3)},
\end{align}
since
\begin{align}
\Vert(\bs{\CC}^\delta-\bs{\CC}^0)\{\Vb\}\Vert_{H^{s-1}(\Stwo,\R^3)}^2 \leq 2\epsilon^2\norm{\Vb}_{H^s(\Stwo,\R^3)}.
\end{align}
Thus, the norm is bounded by the first term that is arbitrarily small and the second term that converges to $0$ as $n\to\infty$ or equivalently as $\delta\to0$; for definiteness, choose
\begin{equation*}
\varepsilon = \frac{1}{n(n+1) + 1},\quad \delta = \frac{4\epsilon}{3[n(n+1) + 1]} = \frac{4}{3[n(n+1) + 1]^2},\quad\hbox{and take}\quad n\to\infty.
\end{equation*}
Again, convergence of the second term follows from the completeness of vector spherical harmonics in $H^s(\Stwo,\R^3)\hookrightarrow L^2(\Stwo,\R^3)$.
\end{proof}

\section{Discussion}

We introduced a nonlocal vector calculus on the sphere based on weakly singular integral operators. With spherical harmonics as bases, our nonlocal operators are diagonal, enabling the proof of a nonlocal Stokes theorem. We established strong convergence to local operators as the interaction range approaches zero. Our approach aligns with our previous work on nonlocal diffusion \cite{montanelli2018b} and extends the nonlocal calculus for Euclidean domains \cite{du2013a} to the sphere. Future work includes extending this framework to higher-dimensional spheres, oblate/prolate spheroids, or arbitrary closed surfaces. Challenges for the latter include the fact that the normal $\bs{\nu}(\xb)$ is not equal to $\xb$, the distance is not solely a function of the dot product, and there is no natural basis in general.

Our nonlocal vector calculus enables the numerical solution of nonlocal advection and advection-reaction-diffusion equations on the sphere, relevant for atmospheric modeling. Since Charney et al.\cite{charney1950}, local advection equations have been widely studied, with benchmark tests from Williamson et al.\cite{williamson1992} and others \cite{nair2008, nair2010}. Advection-reaction-diffusion models describe phenomena like atmospheric chemical transport \cite{pudykiewciz2006}, and nonlocal extensions could provide new insights, following stability studies by Krause et al.~\cite{krause2018}. Our discretization should exhibit asymptotic compatibility \cite{du2013b, du2016, du2017b}, ensuring uniform convergence as $\delta$ goes to zero. Diagonal scalings can be computed in $\OO(\ell)$ flops using Clenshaw--Curtis quadrature, with $\OO(\ell^2)$ complexity if all integrals are needed. The dominant cost remains the spherical harmonic transform at $\OO(\ell^2\log^2\ell)$ \cite{slevinsky2018b, slevinsky2017b, slevinsky2019b}, though fast multipole methods could accelerate computations to $\OO(\ell\log(\epsilon^{-1}))$ \cite{alpert1991, keiner2011}. For time-stepping, exponential integrators like ETDRK4, shown effective by Montanelli and Bootland \cite{montanelli2020b}, could be utilized, with further comparisons available in \cite{montanelli2018a, montanelli2017phd}.

Finally, nonlocal modeling and integral operators play a crucial role in Scientific Machine Learning \cite{montanelli2025a}. Notably, Fourier neural operators \cite{li2021} and other neural operators rely on efficient computations of kernel integrals. For periodic problems, Fourier symbols can be computed efficiently \cite{du2017a} and could be incorporated into such machine learning models. In the future, we plan to investigate problems on the sphere involving nonlocal operators.

\bibliographystyle{siam}
\bibliography{/Users/montanelli/Seafile/WORK/ACADEMIA/BIBLIOGRAPHY/_references.bib}

\end{document}